\documentclass[a4paper,10pt]{article}
\usepackage{stmaryrd}
\usepackage{amscd}
\usepackage{bbding}
\usepackage{amsfonts}
\usepackage{cite}
\usepackage{mathrsfs}
\usepackage{graphicx,latexsym,amsmath,amssymb,amsthm,enumerate}
\usepackage{amsmath,bm}
\usepackage{cite}
\usepackage{color}
\theoremstyle{definition}
\newtheorem{theorem}{Theorem}[section]
\newtheorem{lemma}{Lemma}[section]
\newtheorem{problem}{Problem}[section]

\newtheorem{remark}{Remark}[section]

\hfuzz=\maxdimen
\begin{document}
\title{{\Large\bf{A new class of differential quasivariational inequalities with an application to a quasistatic viscoelastic frictional contact problem}}\thanks{This work was supported by the National Natural Science Foundation of China (11671282, 11771067, 12171339, 12171070).}}
\author{Xu Chu$^a$, Tao Chen$^a$, Nan-jing Huang$^a$\footnote{Corresponding author. E-mail address: nanjinghuang@hotmail.com; njhuang@scu.edu.cn}
and Yi-bin Xiao$^b$
 \\{\small a. Department of Mathematics, Sichuan University, Chengdu, Sichuan, P.R. China}\\
 { \small b. School of Mathematical Sciences, University of Electronic Science and Technology of China, Chengdu, P.R. China}}
\date{ }
\maketitle
\begin{flushleft}
\hrulefill\\
\end{flushleft}
{\bf Abstract}: The overarching goal of this paper is to introduce and investigate a new nonlinear system driven by a nonlinear differential equation, a history-dependent quasivariational inequality, and a parabolic variational inequality in Banach spaces. Such a system can be used to model quasistatic frictional contact problems for viscoelastic materials with long memory, damage and wear. By using the Banach fixed point theorem, we prove an existence and uniqueness theorem of solution for such a system under some mild conditions. As a novel application, we obtain a unique solvability of a quasistatic viscoelastic frictional contact problem with long memory, damage and wear.

\section{Introduction}

 Assume that $V, X, Y$ and $W$ are separable and reflexive Banach spaces, $V^*$ and $Y^*$ are the dual spaces of $V$ and $Y$, respectively. Let $Y\subset Y_1\subset Y^*$, where $Y_1$ is a separable Hilbert space. Suppose that $K_V$ and $K_Y$ are closed, convex and nonempty subsets of $V$ and $Y$, respectively. Let $I:=[0,T]$, where $T>0$ is a constant. In this paper, we are interested in studying a new class of differential quasivariational inequalities (DQVIs) with the following form: find $u:I\rightarrow K_V$, $\zeta:I\rightarrow K_Y$ and $w:I\rightarrow W$ such that, for all $t\in I$,
\begin{equation}\label{1.1}
\left\{\begin{array}{l}
\dot{w}(t)=F(t,w(t),\dot{u}(t)),\\
\left\langle A(t,u(t)) +\int_0^t B(t-s,u(s),\zeta(s))ds+C(t,\dot{u}(t)),v-\dot{u}(t) \right\rangle_{V^*\times V}\\
\quad \mbox{}+j(w(t) ,\dot{u}(t),v)-j(w(t) ,\dot{u}(t) ,\dot{u}(t))\geq \langle f(t) ,v-\dot{u}(t) \rangle_{V^*\times V},\quad\forall v\in K_V ,\\
\langle \dot{\zeta}(t) ,\eta-\zeta(t) \rangle_{Y_1}+a(\zeta(t) ,\eta-\zeta(t) )\geq\langle\phi(t,u(t) ,\zeta(t) ),\eta-\zeta(t) \rangle_{Y_1},\quad\forall\eta\in K_Y,\\
u(0)=u_0, \;w(0)=w_0, \; \zeta(0)=\zeta_0.\\
\end{array}\right.
\end{equation}

Clearly, if $C,w$ and $\zeta$ are omitted, $K_V=V$, $A(t,u(t))=A(u(t))$ and $j(w(t) ,\dot{u}(t),v)=j({u}(t),v)$, then \eqref{1.1} is reduced to the following problem: find $u:I\rightarrow V$ such that, for all $t\in I$,
\begin{equation*}
\left\{\begin{array}{l}
\left\langle A(u(t)) +\int_0^t B(t-s,u(s))ds,v-\dot{u}(t) \right\rangle_{V^*\times V}\\
\quad+j(u(t),v)-j(u ,\dot{u}(t))\geq \langle f(t) ,v-\dot{u}(t) \rangle_{V^*\times V},\quad\forall v\in V ,\\
u(0)=u_0,\\
\end{array}\right.
\end{equation*}
which is the time-dependent quasivariational inequality with the history-dependent operator considered by Kasri and  Touzaline \cite{A2018}.

The research on differential variational inequalities (DVIs) has a long history (see, for example, the excellent survey due to Brogliato and Tanwani \cite{Brog2020} and the references therein). It is well known that Pang and Stewart \cite{Pang2008} are the first systematically to consider DVIs in finite-dimensional Euclidean spaces. Since then, various theoretical results, approximating algorithms and real applications have been investigated extensively for classical DVIs and DQVIs under different conditions in the literature; for instance we refer the reader to \cite{Liu2018,Liu2021,Liu2021+,Weng2021,Weng2021+,Zeng2018,ChenX2014,Gwinner2013,LiXS2010,LiW2015,LiW2017,Wang2017,Zeng2021,Zeng2021+} and the references therein.

It is worth mentioning that contact mechanics has important applications in daily life and industry, such as coupling devices, bearings, ultrasonic welding and many others.  After nearly 40 years of research, contact mechanics has formed a relatively complete set of mathematical theory. It is well known that the variational inequalities and quasivariational inequalities as well as hemivariational inequalities are the most important mathematical tools to obtain the existence and uniqueness of the solutions for various problems arising in contact mechanics \cite{Capatina2014,Gasi2015,Han2001,AS2007,MOS2015,Sofonea2012,Kulig2018, T1999,Fr1996,Fr1995}. 
In order to describe the long memory property of materials such as rock and rubber, Sofonea and Matei \cite{Sofonea2011} introduced and studied a new class of history-dependent quasivariational inequalities. Moreover, by relaxing the contact condition in \cite{Sofonea2011}, Sofonea and Xiao \cite{Sofonea2016} considered a new class of quasivariational inequalities involving two history-dependent operators.  On the other hand, applying Archard's wear laws and Coulomb's friction laws, Andrews et al. \cite{KT1997} obtained the existence and uniqueness of the solution of a dynamic thermoviscoelastic contact problem. Furthermore, Sofonea et al. \cite{SofoneaF2016} proposed a new mathematical model to capture frictional contact problem with wear. Recently, Chen et al. \cite{Chen2020} introduced a hyperbolic quasi-variational inequality to characterize a dynamic viscoelastic contact problem with friction and wear. Very recently, in order to model an elastic frictional contact problem with long memory, damage and wear, Chen et al. \cite{Chen2021} considered a new class of differential nonlinear system driven by a differential equation, a history-dependent hemivariational inequality and a parabolic variational inequality in Banach spaces. Nevertheless, in the study of quasistatic viscoelastic frictional contact problems, it is necessary to consider the properties of viscoelastic materials with long memory, damage and wear (see Section 4 for more details). To the best of authors' knowledge, there are few works considering quasistatic frictional contact problems for viscoelastic materials with long memory, damage and wear. Thus, it would be important and interesting to investigate  DQVIs \eqref{1.1} which can be used to model quasistatic frictional contact problems for viscoelastic materials with long memory, damage and wear.

The overarching goal of this paper is to introduce and investigate a new nonlinear system driven by a nonlinear differential equation, a history-dependent quasivariational inequality, and a parabolic variational inequality in Banach spaces, which can be used to describe the quasistatic frictional contact problems for viscoelastic materials with long memory, damage and wear. The main contributions of this paper are twofold. One is to deliver some sufficient conditions for ensuring the existence and uniqueness of solution to DQVIs \eqref{1.1}. The other is to show the unique solvability result for a new quasistatic viscoelastic frictional contact problem with long memory, damage and wear by applying the obtained results for DQVIs \eqref{1.1}.

The rest of the paper is structured as follows. The next section recalls some known definitions and lemmas. After that in Section 3, we show the existence and uniqueness of solution for DQVIs \eqref{1.1} under some mild conditions by employing the Banach fixed point theorem. In Section 4, we provide a novel application of our abstract results to a quasistatic viscoelastic frictional contact problem with long memory, damage and wear.

\section{Preliminaries}

Let $(X,\|\cdot\|_X)$ be a real Banach space with its dual $X^*$ and $\langle\cdot,\cdot\rangle_{X^*\times X}$ denote the duality pairing between $X^*$ and $X$.
In this section, we recall some known definitions and lemmas which will be used to obtain our main results (see \cite{MOS2013,NP1994} for more details).

A functional $j(u):X\rightarrow\mathbb{R}$ is called lower semicontinuous if and only if for any convergence sequence $\{u_n\}_{n=1}^{\infty}\subset X$ satisfying $u_n\rightarrow u\in X$, one has $\lim\inf_{n\rightarrow\infty}{j(u_n)}\geq j(u)$. Let $j: X\rightarrow \mathbb{R}\cup \{+\infty\}$ be a lower semicontinuous convex functional with domain $D(j)=\{u\in X\:|\:j(u)<+\infty\}$ and $u \in D(j)$. Then there exists $u^*\in X^*$ such that $j(v)-j(u)\geq \langle u^*, v-u\rangle_{X^*\times X}$ holds for all $v\in X$.
The set of all such $u^*\in X^*$ is called the convex subdifferential of $j$ at $u$ and we denote it by $\partial j(u)$.

For a set-valued operator $A:X\rightarrow2^{X^*}$, the graph of $A$ is denoted by $G(A)$, i.e.,
$$G(A):=\{(u,u^*)\in X\times X^*\:|\:u^*\in A(u)\}.$$

A set-valued operator $A:X\rightarrow2^{X^*}$ is called monotone if
$$\langle u^*-v^*,u-v\rangle_{X^*\times X}\geq0,\;\forall(u,u^*),\;(v,v^*)\in G(A).$$
Moreover, a  monotone operator $A$ is called maximal monotone if for any $(u,u^*) \in X\times X^*$ satisfying
$$\langle u^*-v^*,u-v\rangle_{X^*\times X}\geq0, \quad \forall (v,v^*)\in G(A),$$
one has $(u,u^*)\in G(A)$.

A functional $j: X \rightarrow \mathbb{R} \cup\{\infty\}$ is called proper if $j(v)>-\infty$ for all $v\in X$ and there
exists a point $u\in X$ such that $j(u) < +\infty$. For a proper, convex and lower semicontinuous functional $j: X \rightarrow \mathbb{R} \cup\{\infty\}$, it is well known that $\partial j: X\to 2^{X^*}$ is maximal monotone.

At the end of this section, we recall two known results which will be used to obtain our main results.

\begin{lemma}\label{lemma_hisdep}\cite[Proposition 3.1]{Sofonea2012}
Let $\Lambda: C([0, T] ; X) \rightarrow C([0, T] ; X)$ be an operator satisfying the following property: there exists a constant $h>0$ such that
$$
\begin{array}{l}
\left\|\Lambda u_{1}(t)-\Lambda u_{2}(t)\right\|_{X} \leq
h\int_{0}^{t}

\left\|u_{1}(s)-u_{2}(s)\right\|_X ds, \quad \forall u_{1}, u_{2} \in C([0, T] ; X), \; t \in[0, T].
\end{array}
$$
Then there exists a unique element $u^{*} \in C([0, T] ; X)$ such that $\Lambda u^{*}=u^{*}$.
\end{lemma}

\begin{lemma}\label{l2.2}\cite[Lemma 3.3]{LiuSofonea2019}
Let $V$ be a reflexive Banach space and $K$ be a closed convex nonempty subset of $V$. Suppose that $A: K \to V^{*}$ and $\varphi:K\times K \rightarrow \mathbb{R}$ satisfy the following hypotheses:
\begin{itemize}
\item[H(A):] $A: K \rightarrow V^{*}$ is strongly monotone Lipschitz continuous, i.e.,
\begin{itemize}
\item[(a)] $\langle Au_1-Au_2,u_1-u_2\rangle_{V^*\times V}\ge m\|u_1-u_2\|^2_V$ for all $u_1,u_2 \in K$ with $m>0$;
\item[(b)] $\| Au_1-Au_2\|_{V^*} \le L\|u_1-u_2\|_V$ for all $u_1,u_2 \in K$ with $L>0$.
\end{itemize}
\item[H($\varphi$):] $\varphi: K\times K \rightarrow \mathbb{R} $ is such that
\begin{itemize}
\item[(a)] for any $u \in K, \varphi(u,\cdot)$ is convex and lower semicontinuous on $K$;
\item[(b)] for any $u_1,u_2,v_1,v_2\in K $, there exists $\beta>0$ such that
$$ \varphi(u_1,v_2)-\varphi(u_1,v_1)+\varphi(u_2,v_1)-\varphi(u_2,v_2)\leq\beta\|u_1-u_2\|_V\|v_1-v_2\|_V.$$
\end{itemize}
\end{itemize}
If $m>\beta$, then for each $f\in V^*$, there exists a unique element $u\in K$ such that
$$\langle Au,v-u\rangle_{V^*\times V}+\varphi(u,v)-\varphi(u,u)\geq\langle f,v-u\rangle_{V^*\times V}, \quad \forall v\in K.$$
\end{lemma}

\section{Unique solvability for DQVIs \eqref{1.1}}
\setcounter{equation}{0}
In this section, we provide some sufficient conditions for ensuring the existence and uniqueness of solution for DQVIs \eqref{1.1}. To this end, we consider the Gelfand triplet of Banach spaces $\left(V, H, V^{*}\right)$ which have compact and dense embeddings. We also need the following assumptions.
\begin{itemize}
\item[H(A):]\label{assume:A_2}  The operator $A:I\times V\rightarrow V^*$ satisfies
\begin{itemize}
\item[(a)]$A(\cdot, v)$  is continuous on $I$ for any given $v\in V$.
\item[(b)]$A(t,\cdot)$ is Lipschitz continuous with  $ L_A>0$  on $V$ for any given $t\in I$, i.e.,
$$\left\|A(t, u_{1})-A(t, u_{2})\right\|_{V^*}\leq L_{A}\left\|u_{1}-u_{2}\right\|_{V}, \quad \forall (t,u_1,u_2)\in I\times V\times V.$$
\end{itemize}
\item[H(B):] The operator $B: I\times V\times Y\rightarrow V^*$ satisfies
\begin{itemize}\label{assume:B_2}
\item[(a)]$B(\cdot, v,\zeta)$  is continuous on $I$ for any given $v \in V$ and $\zeta \in Y$.
\item[(b)]$B(t,\cdot,\cdot)$ is Lipschitz continuous with $L_B>0$ on $V\times Y$ for any given $t\in I$, i.e.,
   $$\left\| B(t, u_{1},\zeta_1)-B(t, u_{2},\zeta_2)\right\|_{V^*} \leq L_{B}(\|u_{1}-u_{2}\|_{V}+\|\zeta_1-\zeta_2\|_Y),\quad\forall t\in I,\; \forall u_1,u_2 \in V,\;\forall \zeta_1,\zeta_2\in Y.$$
\item[(c)] There exists $\rho\in  L^2(I;\mathbb{R}^+)$ such that
$$\|B(t,u,\zeta)\|_{V^*}\leq \rho(t)(\|\zeta\|_Y+\|u\|_V), \quad \forall (t,u,\zeta)\in I \times V \times Y.$$
\end{itemize}
\item[H(C):]\label{assume:A_2}  The operator $C:I\times V\rightarrow V^*$ satisfies
\begin{itemize}
\item[(a)]$C(t,\cdot)$ is Lipschitz continuous with  $ L_{C1}>0$  on $V$ for any given $t\in I$, i.e.,
$$\left\|C(t, u_{1})-C(t, u_{2})\right\|_{V^*}\leq L_{C1}\left\|u_{1}-u_{2}\right\|_{V}, \quad \forall (t,u_1,u_2)\in I\times V\times V.$$
\item[(b)]$C(\cdot,u)$ is Lipschitz continuous with  $ L_{C2}>0$  on $I$ for any given $u\in V$, i.e.,
$$\left\|C(t_{1}, u)-C(t_{2}, u)\right\|_{V^*}\leq L_{C2}\left\|t_1-t_2\right\|_{V}, \quad \forall (t_1,t_2,u)\in I\times I\times V.$$
\item[(c)]$C(t,\cdot)$ is strong monotone with  $ L_C>0$  on  $V$ for any given $t\in I$, i.e.,
$$\left\langle C(t, u_{1})-C(t, u_{2}), u_{1}-u_{2}\right\rangle_{V^*\times V} \geq m_{C}\|u_{1}-u_{2}\|_{V}^2,\quad \forall (t,u_1,u_2)\in I\times V\times V.$$
\end{itemize}
\item[H(j):] The functional $j: W\times V\times V\rightarrow \mathbb{R}$ satisfies
\begin{itemize}\label{assume:j_2}
\item[(a)]$j(w, u, \cdot)$ is convex proper and lower semicontinuous on  $V$ for any given $(w,u)\in W\times V$.
\item[(b)] There exist $\alpha_0>0$ \ and $\alpha_1>0 $ such that
    \begin{align*}
    &\; j\left(w_{1}, u_{1}, v_{2}\right)-j\left(w_{1}, u_{1}, v_{1}\right)+j\left(w_{2},u_{2},v_{1}\right)-j\left(w_{2}, u_{2},v_{2}\right)\\
    \leq&\; \alpha_0\left\|w_{1}-w_{2}\right\|_{W}\left\|v_{1}-v_{2}\right\|_{V}+\alpha_1\left\|u_{1}-u_{2}\right\|_{V}
    \left\|v_{1}-v_{2}\right\|_{V},\quad \forall w_1, w_2\in W,\;\forall u_{1}, u_{2}, v_{1}, v_{2} \in V.
    \end{align*}
 \end{itemize}
\item[H(F):] The operator $F: I\times W\times V\rightarrow W$ satisfies
\begin{itemize}\label{assume:F_2}
\item[(a)]$F(\cdot, w,v)$ is continuous on $I$  for any given $(w,v)\in W\times V$.
\item[(b)] $F(t,\cdot,\cdot)$  is Lipschitz continuous with $L_F >0$ on $V\times Y$ for any given $t\in I$, i.e.,
        $$\left\| F(t,w_1, u_{1})-F(t,w_2, u_{2})\right\|_{W} \leq L_{F}(\|u_{1}-u_{2}\|_{V}+\|w_1-w_2\|_W), \quad \forall t\in I,\; \forall w_1,w_2\in W,\;\forall u_1,u_2 \in V.$$
\end{itemize}
 \item[H($\phi$):] The operator $\phi: I\times V\times Y\rightarrow Y_1$ satisfies
\begin{itemize}\label{assume:phi_2}
\item[(a)] $\phi(t,\cdot,\cdot)$  is Lipschitz continuous with $L_\phi>0$ on $V\times Y$ for any given $t\in I$, i.e.,
        $$\left\| \phi(t,u, \zeta)-\phi(t,v, \eta)\right\|_{Y_1} \leq L_{\phi}(\|u-v\|_{V}+\|\zeta-\eta\|_{Y_1}), \quad \forall t\in I,\; \forall u,v \in V,\; \forall \zeta,\eta\in Y.$$
\item[(b)]$\phi(\cdot,0_V,0_Y)\in L^2(I; Y_1).$
 \end{itemize}
\item[H(a):] The functional $a: Y\times Y\rightarrow \mathbb{R}$ satisfies
\begin{itemize}\label{assume:a_2}
\item[(a)] $a(\cdot,\cdot)$  is a continuous bilinear symmetric coercive functional and there exist $a_1\in\mathbb{R}$ and $a_2>0$ such that
        $$ a(\eta, \eta)+a_1\|\eta\|_{Y_1}^2\geq a_2\|\eta\|_Y^2,\quad \forall \eta\in Y.$$
 \end{itemize}
\end{itemize}

The main result of this section can be stated as follows.

\begin{theorem}\label{t3.1}
Suppose that assumptions H(A), H(B), H(C), H(j), H(F), H($\phi$), H(a) hold and $ m_C>\alpha_1$. Then DQVIs \eqref{1.1} has a unique solution
$(\zeta,u,w)\in  (H^1(I;Y_1)\cap L^2(I;Y))\times C^1(I;K_V)\times C^1(I;W)$.
\end{theorem}

We divide the proof of Theorem \ref{t3.1} into a series of lemmas. To this end, we first consider the following auxiliary problem.
\begin{problem}\label{problem_aux_1}
For any given $w\in C(I;W)$, find $\dot{u}_w:I\rightarrow K_V$ such that, for all $t\in I$ and
\begin{align}\label{eq_pro1}
&\quad\langle C(t,\dot{u}_w(t)),v-\dot{u}_w(t)\rangle_{V^*\times V} +j(w(t),\dot{u}_w(t),v)-j(w(t),\dot{u}_w(t),u\dot{}_w(t))\nonumber\\
&\geq \langle f(t),v-\dot{u}_w(t)\rangle_{V^*\times V}, \quad
\forall v\in K_V.
\end{align}
\end{problem}

\begin{lemma}\label{Lemma_1}
Assume that H(C) and H(j) hold and $ m_C>\alpha_1$. Then for any given $w\in C(I;W)$ and $f\in C(I;V^*)$, Problem \ref{problem_aux_1} has a unique solution $\dot{u}_w\in C(I,K_V)$.
\end{lemma}
\begin{proof}
For any fixed $t$, define an operator $A:V\rightarrow V^*$ and a function $\varphi:V\times V\rightarrow\mathbb{R} $ by setting
$$ Au=C(t,\dot{u}(t)),\quad \varphi(u,v)=j(w(t),\dot{u}(t),v),\quad\forall u,v\in K, t\in I.$$
Clearly, $A$ and $\varphi$ satisfy all the assumptions of Lemma \ref{l2.2}. Thus,  it follows from Lemma \ref{l2.2} that Problem \ref{problem_aux_1} has a unique solution $\dot{u}_w(t)\in K_V$. Next we prove that $\dot{u}_w\in C(I,K_V)$.  For $t_1,t_2\in I$, denote $\dot{u}_{w}(t_i)=\dot{u}_i,\;f(t_i)=f_i,\;w(t_i)=w_i$ with $i=1,2$. Thus, inequality \eqref{eq_pro1} yields
\begin{eqnarray}\label{test1}
\langle C(t_1,\dot{u}_1),\dot{u}_2-\dot{u}_1\rangle_{V^*\times V} + j(w_1,\dot{u}_1(t),\dot{u}_2)-j(w_1,\dot{u}_1,\dot{u}_1 )\geq\langle f_1,\dot{u}_2-\dot{u}_1\rangle
\end{eqnarray}
and
 \begin{eqnarray}\label{test2}
\langle C(t_2,\dot{u}_2),\dot{u}_1-\dot{u}_2\rangle_{V^*\times V} + j(w_2,\dot{u}_2(t),\dot{u}_1)-j(w_2,\dot{u}_2,\dot{u}_2 )\geq\langle f_2,\dot{u}_1-\dot{u}_2\rangle.
\end{eqnarray}
Adding \eqref{test1} to \eqref{test1}, one has
\begin{eqnarray*}
&&\langle C(t_1,\dot{u}_1)-C(t_2,\dot{u}_2),\dot{u}_1-\dot{u}_2\rangle_{V^*\times V} \nonumber\\
&\leq& j(w_1,\dot{u}_1(t),\dot{u}_2)-j(w_1,\dot{u}_1,\dot{u}_1 )+ j(w_2,\dot{u}_2(t),\dot{u}_1)\nonumber\\
&&\mbox{}-j(w_2,\dot{u}_2,\dot{u}_2 )+\langle f_1-f_2,u_1-u_2\rangle
\end{eqnarray*}
and so
\begin{eqnarray}\label{test_sum}
&&\langle C(t_1,\dot{u}_1)-C(t_1,\dot{u}_2),\dot{u}_1-\dot{u}_2\rangle_{V^*\times V} \nonumber\\
&\leq& j(w_1,\dot{u}_1(t),\dot{u}_2)-j(w_1,\dot{u}_1,\dot{u}_1 )+ j(w_2,\dot{u}_2(t),\dot{u}_1)\nonumber\\
&&\mbox{}-j(w_2,\dot{u}_2,\dot{u}_2 )+\langle f_1-f_2,\dot{u}_1-\dot{u}_2\rangle+\langle C(t_2,\dot{u}_2)-C(t_1,\dot{u}_2),\dot{u}_1-\dot{u}_2\rangle_{V^*\times V}.
\end{eqnarray}
By conditions H(j)(b) and H(C)(b)(c), inequality \eqref{test_sum} becomes
\begin{eqnarray}\label{test_sum2}
(m_C-\alpha_1)\|\dot{u}_1-\dot{u}_2\|
\leq \alpha_0\|w_1-w_2\|_W+\|f_1-f_2\|_{V^*} +L_{C2}|t_1-t_2|.
\end{eqnarray}
Since $w\in C(I;W)$ and $f\in C(I;V^*)$, it follows from \eqref{test_sum2} that
$$\lim_{t_1\rightarrow t_2}\|\dot{u}_w(t_1)-\dot{u}_w(t_2)\|_V=0,$$
i.e., $\dot{u}_w\in C(I,K_V)$. This finishes the proof.
\end{proof}
\begin{lemma}\label{Corollary_1}
Suppose that H(A), H(B), H(C), H(j) hold and $ m_C>\alpha_1$. Then for any given $\zeta\in H^1(I;Y_1)\cap L^2(I;Y)$, $w\in C(I;W)$, and $f\in C(I;V^*)$, the following problem
\begin{equation}\label{eq_col}
\left\{\begin{array}{l}
\text { find } u_{w\zeta}:(0, T) \rightarrow K_V \text { such that } \\
\left\langle A(t,u_{w\zeta}(t))+\int_0^tB(t-s,u_{w\zeta}(s),\zeta(s))ds+C(t,\dot{u}_{w\zeta}(t)),v-\dot{u}_{w\zeta}(t)\right\rangle_{V^*\times V}\\ \quad \mbox{}+j(w,\dot{u}_{w\zeta}(t),v)-j(w,\dot{u}_{w\zeta}(t),\dot{u}_{w\zeta}(t))
\geq \langle f(t),v-\dot{u}_{w\zeta}(t)\rangle_{V^*\times V}, \quad \forall t \in I,\; \forall v \in K_V, \\
\text { where }  u_{w\zeta}(0)=u_0
\end{array}\right.
\end{equation}
has a unique solution $u_{w\zeta}\in C^1(I;K_V)$.
\end{lemma}

 \begin{proof}
Let us fix $\eta \in C(I; K_V )$. We consider the following auxiliary for \eqref{eq_col}.
\begin{equation}\label{eq_col4_1}
\left\{\begin{array}{l}
\text { find } u_{w\zeta\eta}:(0, T) \rightarrow K_V \text { such that } \\
\langle C(t,\dot{u}_{w\zeta\eta}(t)),v-\dot{u}_{w\zeta\eta}(t)\rangle_{V^*\times V} +j(w(t),\dot{u}_{w\zeta\eta}(t),v)-j(w(t),\dot{u}_{w\zeta\eta}(t),\dot{u}_{w\zeta\eta}(t))\\
\quad \mbox{}\geq \langle f_\eta(t),v-\dot{u}_{w\zeta\eta}(t)\rangle_{V^*\times V}, \quad \forall t \in I, \; \forall v \in K_V,\\
\text { where }  u_{w\zeta\eta}(0)=u_0,
\end{array}\right.
\end{equation}
where $f_\eta$ is defined by
$$f_\eta(t):=f(t)-\int_0^tB\left(t-s,\int_0^s\eta(s)ds,\zeta(s)\right)-A\left(t,\int_0^t\eta(s)ds\right).$$

We first claim that $f_\eta\in C(I;V^*)$. For any given $ t_1, \;t_2\in I$ with $t_1<t_2$, it follows from the definition of $f_\eta$ that
\begin{eqnarray}
&& \|f_\eta(t_2)-f_\eta(t_1)\|_V^* \nonumber\\
 &\leq& \|f(t_1)-f(t_2)\|+ \left\|\int_0^{t_2}B\left(t_2-s,\int_0^s\eta(s)ds,\zeta(s)\right)ds-\int_0^{t_1}B\left(t_1-s,\int_0^s\eta(s)ds,\zeta(s)\right)ds\right\|\nonumber\\
 &&\mbox{}+\left\|A\left(t_2,\int_0^{t_2}\eta(s)ds\right)-A\left(t_2,\int_0^{t_1}\eta(s)ds\right)+A\left(t_2,\int_0^{t_1}\eta(s)ds\right)-A\left(t_1,\int_0^{t_1}\eta(s)ds\right)\right\|\nonumber\\
 &\leq&\|f(t_1)-f(t_2)\|+\int_{t_1}^{t_2}\left\|B\left(t_2-s,\int_0^s\eta(s)ds,\zeta(s)\right)\right\|ds\nonumber\\
 &&\mbox{}+\int_{0}^{t_1}\left\|B\left(t_2-s,\int_0^s\eta(s)ds,\zeta(s)\right)ds-B\left(t_1-s,\int_0^s\eta(s)ds,\zeta(s)\right)\right\|ds\nonumber\\
 &&\mbox{}+\left\|A\left(t_2,\int_0^{t_1}\eta(s)ds\right)-A\left(t_1,\int_0^{t_1}\eta(s)ds\right)\right\|+L_A\left\|\int_{t_1}^{t_2}\eta(s)ds\right\|.
\end{eqnarray}
By condition H(B)(c) and  H\"{o}lder's inequality, we have
\begin{eqnarray}
&&\int_{t_1}^{t_2}\left\|B\left(t_2-s,\int_0^s\eta(s)ds,\zeta(s)\right)\right\|ds\nonumber\\
&\leq&\int_{t_1}^{t_2}|\rho(t_2-s)|\left\|\int_0^s\eta(s)ds\right\|ds+\int_{t_1}^{t_2}|\rho(t_2-s)|\left\|\zeta(s)\right\|ds\nonumber\\
&\leq&T\|\eta\|_{C^(I;V)}\left(\sqrt{t_2-t_1}\right) \left(\int_{t_1}^{t_2}|\rho(t_2-s)|^2ds\right)^{\frac{1}{2}}+ \left(\int_{t_1}^{t_2}|\rho(t_2-s)|^2ds\right)^{\frac{1}{2}}\left(\int_{t_1}^{t_2}\|\zeta(s)\|^2ds\right)^{\frac{1}{2}}\nonumber\\
&\leq&T^\frac{3}{2}\|\eta\|_{C^(I;V)}\|\rho\|_{L^2(t_1,t_2;\mathbb{R}^+)}+\|\rho\|_{L^2(t_1,t_2;\mathbb{R}^+)}\|\zeta(s)\|_{L^2(t_1,t_2;Y)}\nonumber\\
&&\mbox{}\rightarrow0\;\mbox{as $|t_1-t_2|\rightarrow0$}.
\end{eqnarray}
According to condition H(B)(a), one has
\begin{eqnarray}
\int_{0}^{t_1}\left\|B\left(t_2-s,\int_0^s\eta(s)ds,\zeta(s)\right)ds-B\left(t_1-s,\int_0^s\eta(s)ds,\zeta(s)\right)\right\|ds\rightarrow0\;\mbox{as $|t_1-t_2|\rightarrow0$}.
\end{eqnarray}
By the continuity of $f$ and $A(t,\cdot)$, we know that $\lim_{|t_1-t_2|\rightarrow0}\|f_\eta(t_2)-f_\eta(t_1)\|_V^*=0$ and so $f_\eta\in C(I;V^*)$.

Now it follows from Lemma \ref{Lemma_1} that problem  \eqref{eq_col4_1} has a unique solution $\dot{u}_{w\zeta\eta}\in C(I;K_V)$. Thus, we can define an operator $\Lambda:\eta\mapsto  \dot{u}_{w\zeta\eta}$, where $\eta\in C(I;V)$ and $\dot{u}_{w\zeta\eta}(t)$ is the unique solution of problem \eqref{eq_col4_1}.
Next we show that $\Lambda$ has a unique fixed point in $C(I;V)$. In fact, let $\dot{u}_0(t)$ and $\dot{u}_1(t)$ be two solutions of  \eqref{eq_col4_1}. Then
\begin{align}\label{eq_col4_2}
\langle C(t,\dot{u}_i(t)),v-\dot{u}_i(t)\rangle_{V^*\times V} +j(w,\dot{u}_i(t),v)-j(w,u\dot{}_i(t),\dot{u}_i(t))
\geq \langle f_{\eta_i}(t),v-\dot{u}_i(t)\rangle_{V^*\times V},\; i=0,1
\end{align}
for all $v \in K_V$  and all $t\in I$. Testing \eqref{eq_col4_2} with $\dot{u}_{1-i}(t)$, we obtain
\begin{eqnarray}\label{eq_col4_3}
&&\langle C(t,\dot{u}_0(t))-C(t,\dot{u}_1(t)),\dot{u}_0-\dot{u}_1(t)\rangle_{V^*\times V}\nonumber\\
&\leq& j(w,\dot{u}_0(t),\dot{u}_1(t))+j(w,\dot{u}_1(t),\dot{u}_0(t))-j(w,\dot{u}_1(t),\dot{u}_1(t))\nonumber\\
&&\mbox{}-j(w,\dot{u}_0(t),\dot{u}_0(t))+\langle f_{\eta_1}- f_{\eta_0},\dot{u}_0(t)-\dot{u}_1(t)\rangle_{V^*\times V}
\end{eqnarray}
for all $t \in I$.  According to conditions H(A)(b), H(B)(b) and H(j)(b), inequality \eqref{eq_col4_3} yields
\begin{eqnarray}
&& m_C\| \dot{u}_0(t)-\dot{u}_1(t)\|^2_{ V}\nonumber\\
&\leq&\alpha_1\|\dot{u}_0(t)-\dot{u}_1(t)\|^2_V+\|f_{\eta_1}(t)- f_{\eta_0}(t)\|_{V^*}\|\dot{u}_0(t)-\dot{u}_1(t)\|_V\nonumber\\
&\leq&\alpha_1\|\dot{u}_0(t)-\dot{u}_1(t)\|^2_V++\left\|A\left(t,\int_0^t\eta_0(s)ds\right)-A\left(t,\int_0^t\eta_1(s)ds\right)\right\|_V\|\dot{u}_0(t)-\dot{u}_1(t)\|_V\nonumber\\
&&\mbox{}+\left\|\int_0^{t_2}B\left(t-s,\int_0^s\eta_0(s)ds,\zeta(s)\right)ds-\int_0^{t_1}B\left(t-s,\int_0^s\eta_1(s)ds,\zeta(s)\right)ds\right\|_V\|\dot{u}_0(t)-\dot{u}_1(t)\|_V\nonumber\\
&\leq&\alpha_1\|\dot{u}_0(t)-\dot{u}_1(t)\|^2_V+L_A\int_0^t\|\eta_0(s)-\eta_1(s)\|_Vds\|\dot{u}_0(t)-\dot{u}_1(t)\|_V\nonumber\\
&&\mbox{}+L_B\int_0^t\int_0^t\|\eta_0(s)-\eta_1(s)\|_Vdlds\|\dot{u}_0(t)-\dot{u}_1(t)\|_V\nonumber\\
&\leq&\alpha_1\|\dot{u}_0(t)-\dot{u}_1(t)\|^2_V+(L_A+TL_B)\int_0^t\|\eta_0(s)-\eta_1(s)\|_Vds\|\dot{u}_0(t)-\dot{u}_1(t)\|_V
\end{eqnarray}
for all $t \in I$. Hence
\begin{align*}
\| \dot{u}_0(t)-\dot{u}_1(t)\|_{ V}\leq \frac{L_A+TL_B}{ (m_C-\alpha_1)}\int_0^t\|\eta_1(s)-\eta_2(s)\|_{V}ds.
\end{align*}
From Lemma \ref{lemma_hisdep} we can see that $\Lambda$ has a unique fixed point $\dot{u}_{w\zeta}\in C(I;K_V)$. Since  $u_{w\zeta}(0)=u_0$, we know that \eqref{eq_col} has a unique solution $u_{w\zeta}\in C^1(I;K_V)$.
This completes the proof.
\end{proof}
\begin{remark}\label{r3.1}
We need some useful inequalities for the solution $u_{w\zeta}$ of problem \eqref{eq_col}. Let $w_1,\;w_2\in C(I;W)$ and $\zeta_1,\;\zeta_2\in H^1(I;Y_1)\cap L^2(I;Y)$. Replacing $w=w_i$ and $\zeta=\zeta_i$ with $i=1,\;2$ in problem \eqref{eq_col} and applying H\"{o}lder's inequality, we have
\begin{eqnarray}
(m_C-\alpha_1)\|\dot{u}_{w_1\zeta_1}(t)-\dot{u}_{w_2\zeta_2}(t)\|
&\leq&
L_B\int_0^t\int_0^t\|\dot{u}_{w_1\zeta_1}(s)-\dot{u}_{w_2\zeta_2}(s)\|ds+L_B\int_0^t\|\zeta_1(s)-\zeta_2(s)\|ds\nonumber\\
&&\mbox{}+\alpha_0\|w_{1}(t)-w_{2}(t)\|+L_A\int_0^t\|\dot{u}_{w_1\zeta_1}(s)-\dot{u}_{w_2\zeta_2}(s)\|ds\nonumber\\
&\leq&
(L_A+L_Bt)\int_0^t\|\dot{u}_{w_1\zeta_1}(s)-\dot{u}_{w_2\zeta_2}(s)\|ds+\mbox{}+\alpha_0\|w_{1}(t)-w_{2}(t)\|\nonumber\\
&&\mbox{}+L_B\sqrt t\left(\int_{0}^{t}\|\zeta_1(s)-\zeta_2(s)\|^2ds\right)^{\frac{1}{2}}ds\
\end{eqnarray}
and so
\begin{eqnarray}
\|\dot{u}_{w_1\zeta_1}(t)-\dot{u}_{w_2\zeta_2}(t)\|&\leq&
\frac{L_A+L_Bt}{m_C-\alpha_1}\int_0^t \|\dot{u}_{w_1\zeta_1}(s)-\dot{u}_{w_2\zeta_2}(s)\|ds+\frac{L_B\sqrt t}{m_C-\alpha_1}\|\zeta_1-\zeta_2\|_{L^2(I;Y_1)}\nonumber\\
&&\mbox{}+\frac{\alpha_0}{m_C-\alpha_1}\|w_{1}(t)-w_{2}(t)\|.
\end{eqnarray}
Applying Gronwall's inequality yields
\begin{eqnarray}
&&\|\dot{u}_{w_1\zeta_1}(t)-\dot{u}_{w_2\zeta_2}(t)\|\nonumber\\
&\leq&\frac{(L_A+L_Bt)\alpha_0}{(m_C-\alpha_1)^2}e^{\frac{2L_At+L_Bt^2}{2(m_C-\alpha_1)}}\int_0^t \|w_{1 }(s)-w_{2 }(s)\|ds+\frac{\alpha_0}{m_C-\alpha_1}\|w_{1}(t)-w_{2}(t)\|\nonumber\\
&&\mbox{}+\left(\frac{L_B \sqrt t}{m_C-\alpha_1}+\frac{2L_B t^\frac{3}{2}(L_A+L_Bt)}{3(m_C-\alpha_1)^2}e^{\frac{2L_At+L_Bt^2}{2(m_C-\alpha_1)}}\right)\|\zeta_1-\zeta_2\|_{L^2(I;Y_1)}.
\end{eqnarray}
\end{remark}

In order to prove Theorem \ref{t3.1}, we also need to introduce the following auxiliary problem.
\begin{problem}\label{problem_aux_2}
Find $u_\zeta:I\rightarrow K_V$ and $w_\zeta:I\rightarrow W$ such that, for all $t\in I$,
\begin{align*}
&\dot{w}_{\zeta}(t)=F(t,w_{\zeta}(t),\dot{u}_{\zeta}(t)),\nonumber\\
&\left\langle A(t,u_{\zeta}(t)) +\int_0^t B(t-s,u_{\zeta}(s)),\zeta(s))ds+C(t,\dot{u}_{\zeta}(t)),v-\dot{u}_{\zeta}(t) \right\rangle_{V^*\times V}\nonumber\\
&\quad \mbox{}+j(w_{\zeta}(t) ,\dot{u}_{\zeta}(t) ,v)-j(w_{\zeta}(t) ,\dot{u}_{\zeta}(t) ,\dot{u}_{\zeta}(t) )\geq \langle f(t) ,v-\dot{u}_{\zeta}(t) \rangle_{V^*\times V},\quad\forall v\in K_V,
\end{align*}
where $u_\zeta(0)=u_0$, $w_\zeta(0)=w_0$ and $\zeta\in C(I;K_Y)$.
\end{problem}
\begin{lemma}\label{lemma_aux_2}
Suppose that H(A), H(B), H(C), H(j) and H(F) hold and $ m_C>\alpha_1$. Then for any given $\zeta\in H^1(I;Y_1)\cap L^2(I;Y)$ and $f\in C(I;V^*)$,
Problem \ref{problem_aux_2} has a unique solution  $(u_\zeta,w_\zeta)\in C^1(I;K_V)\times C^1(I;W)$.
\end{lemma}
\begin{proof}
Define an operator $G:C^1(I;W)\rightarrow C(I;K_V)$ by $G(w)(t)=\dot{u}_{w\zeta}(t)$, where $u_{w\zeta}$ is the solution of problem \eqref{eq_col}. By the proof of Lemma \ref{Corollary_1}, we know that $G$ is well-defined.
Then we need to illustrate that there exists a unique $w_\zeta\in C^1(I;W)$ such that
\begin{align}\label{eq_lem4_2}
\dot{w}_{\zeta}(t)=F(t,w_{\zeta}(t),G(w_\zeta)(t)).
\end{align}
Moreover, $(G(w_\zeta),w_\zeta)\in C(I;K_V)\times C^1(I;W)$ is the unique solution of
Problem \ref{problem_aux_2}.

For this purpose, we consider an operator $\Lambda : C(I;W)\rightarrow C^1(I;K_V)$ defined as follows:
$$\Lambda w_{\zeta}(t)=\int_{0}^{t} F(s,w_{\zeta}(s),G(w_\zeta)(s))ds +w_0, \;\forall t \in[0, T].$$
According to the definition of operator $G$ and conditions H(F)(a)(b) that
$\Lambda w_\zeta\in C^1(I;W)$ when $w_\zeta\in C(I;W)$. Because the fixed point of $\Lambda$ is the
solution of \eqref{eq_lem4_2}, we only need to prove that $\Lambda$ has a unique fixed point
in $C(I;W)$.

Let $w_{\zeta1},\;w_{\zeta2}\in C(I;W)$. We can draw a conclusion from $G$, $\Lambda$ and condition H(F)(b) that
\begin{align}\label{eq_lem4_3}
&\quad \| \Lambda w_{\zeta1}(t)- \Lambda w_{\zeta2}(t)\|\nonumber\\
&=\left\|\int_{0}^{t} F(s,w_{\zeta1}(s),G(w_{\zeta1})(s))ds-\int_{0}^{t} F(s,w_{\zeta2}(s),G(w_{\zeta2})(s))ds\right\|\nonumber\\
&\leq\int_{0}^{t}\left\| F(s,w_{\zeta1}(s),G(w_{\zeta1})(s))ds- F(s,w_{\zeta2}(s),G(w_{\zeta2})(s))\right\| ds.\nonumber\\
&\leq\int_0^tL_F\left(\|w_{\zeta1}(s)-w_{\zeta2}(s)\|+\|G(w_{\zeta1})(s)-G(w_{\zeta2})(s)\|\right)ds.
\end{align}
Denote
$$c_q=\frac{(L_A+L_BT)\alpha_0}{(m_C-\alpha_1)^2}e^{\frac{2L_AT+L_BT^2}{2(m_C-\alpha_1)}}, \quad c_p=\frac{\alpha_0}{m_A-\alpha_1}.$$
Then it follows from Remark \ref{r3.1} that
$$\|G(w_{\zeta1})(s)-G(w_{\zeta2})(s)\|
\leq c_q\int_0^s \|w_{\zeta1}(l)-w_{\zeta2}(l)\|dl
+c_p\|w_{\zeta1}(s)-w_{\zeta2}(s)\|.$$
These inequalities give
\begin{align}\label{eq_lem4_4}
&\quad \| \Lambda w_{\zeta1}(t)- \Lambda w_{\zeta2}(t)\|\nonumber\\
&\leq (L_F+L_F c_p)\int_0^t\|w_{\zeta1}(s)-w_{\zeta2}(s)\|ds
+L_Fc_q\int_0^t\int_0^s \|w_{\zeta1}(l)-w_{\zeta2}(l)\|dlds
\end{align}
and so
\begin{align}\label{eq_lem4_5}
 \| \Lambda w_{\zeta1}(t)- \Lambda w_{\zeta2}(t)\|_\beta\leq
\left(\frac{L_F+L_F c_p}{\beta}\right)\left\|w_{\zeta1}-w_{\zeta2}\right\|_\beta
+\frac{L_Fc_q}{\beta^2} \left\|w_{\zeta1}-w_{\zeta2}\right\|_\beta
\end{align}
with $\beta>0$ and
$\|w\|_{\beta}=\max _{t \in I} e^{-\beta t}\|w(t)\|_{W}$ for all $w \in C(I ; W).$
Thus, $\|\cdot\|_{\beta}$ is an equivalent norm in $C(I ; W)$.
From inequality \eqref{eq_lem4_5}, we know that $\Lambda$ is a contraction operator on $C(I ; W)$ and $C(I ; W)$
endowed with the norm $\|\cdot\|_{\beta}$. Thus, $\Lambda$ has a unique fixed point $w_\zeta\in C(I;W)$ and so $(w_\zeta,G (w_\zeta))$ solves the Problem \ref{problem_aux_2}, which concludes the proof.
\end{proof}

\begin{remark}\label{r3.2}
We need some useful inequalities for the solution $u_{\zeta}$ of Problem \ref{problem_aux_2}. For any given $\zeta_1,\;\zeta_2\in H^1(I;Y_1)\cap L^2(I;Y)$, let $(\dot{u}_{\zeta_i},w_{\zeta_i})\in C(I;V)\times C^1(I;W)$ be the unique solution of Problem \ref{problem_aux_2} with $\zeta=\zeta_i$ for $i=1,\;2.$
From the proof of Lemma \ref{lemma_aux_2}, we have
\begin{eqnarray}\label{remark_es1}
&&\|w_{\zeta_1}(t)-w_{\zeta_2}(t)\|\nonumber\\
&\leq& (L_F+L_F c_p)\int_0^t \|w_{\zeta_1}(s)-w_{\zeta_2}(s)\|ds
+L_Fc_q\int_0^t\int_0^s \|w_{\zeta_1}(l)-w_{\zeta_2}(l)\|dlds
+L_Fc_r T\|\zeta_1-\zeta_2\|_{L^2(I,Y_1)}\nonumber\\
&\leq& (L_F+L_F c_p+L_Fc_qT)\int_0^t \|w_{\zeta_1}(s)-w_{\zeta_2}(s)\|ds+L_Fc_r T\|\zeta_1-\zeta_2\|_{L^2(I,Y_1)}
\end{eqnarray}
and
\begin{eqnarray}\label{remark_es2}
\|\dot{u}_{\zeta_1}(t)-\dot{u}_{\zeta_2}(t)\|
&\leq&  c_q\int_0^t \|w_{\zeta_1}(s)-w_{\zeta_2}(s)\|ds
+c_p\|w_{\zeta_1}(t)-w_{\zeta_2}(t)\|+c_r\|\zeta_1-\zeta_2\|_{L^2(I,Y_1)}.\nonumber\\
&\leq& ( c_qT+c_p) \|w_{\zeta_1}(s)-w_{\zeta_2}(s)\|+c_r\|\zeta_1-\zeta_2\|_{L^2(I,Y_1)},
\end{eqnarray}
where
$$c_r=\frac{L_B \sqrt t}{m_C-\alpha_1}+\frac{2L_B t^\frac{3}{2}(L_A+L_Bt)}{3(m_C-\alpha_1)^2}e^{\frac{2L_At+L_Bt^2}{2(m_C-\alpha_1)}}.$$
By Gronwall's inequality, it follows from \eqref{remark_es1} that
\begin{eqnarray}\label{remark_es3}
\|w_{\zeta_1}(t)-w_{\zeta_2}(t)\|
\leq \left(L_Fc_r T\|\zeta_1-\zeta_2\|_{L^2(I,Y_1)}\right)e^{(L_F+L_F c_p+L_Fc_qT)T}.
\end{eqnarray}
Combining \eqref{remark_es2} and \eqref{remark_es3} yields
\begin{eqnarray*}
\|\dot{u}_{\zeta_1}(t)-\dot{u}_{\zeta_2}(t)\|_V
\leq  \left(( L_Fc_qc_rT^2+L_Fc_pc_rT)e^{(L_F+L_F c_p+L_Fc_qT)T}+c_r\right)\|\zeta_1-\zeta_2\|_{L^2(I,Y_1)}.
\end{eqnarray*}
\end{remark}

Moreover, for the unique solution $(u_\zeta,w_\zeta)\in C^1(I;K_V)\times C^1(I;W)$ of Problem \ref{problem_aux_2}, we need to consider the following auxiliary problem.
\begin{problem}\label{problem_last}
  Find $\zeta:I\rightarrow K_Y$ such that
 \begin{eqnarray}
\langle \dot{\zeta}(t) ,\eta-\zeta(t) \rangle_{Y_1}+a(\zeta(t) ,\eta-\zeta(t) )\geq\langle\phi(t,u_\zeta(t) ,\zeta(t) ),\eta-\zeta(t) \rangle_{Y_1},\quad\forall\eta\in K_Y,
\end{eqnarray}
with $\zeta(0)=\zeta_0\in K_Y.$
\end{problem}
\begin{remark}\label{remark_solu}
We observe that if $\zeta\in H^1(I;Y_1)\cap L^2(I;Y)$ is a unique solution of Problem \ref{problem_last}, then $(\zeta,u_\zeta,w_\zeta)\in  H^1(I;Y_1)\cap L^2(I;Y)\times C^1(I;K_Y)\times C^1(I;W)$ is a unique solution of DQVIs \eqref{1.1}.
\end{remark}

In order to prove Lemma \ref{lemma_last}, we need the following Lemma given in \cite{Han2016}.
\begin{lemma} \cite{Han2016} \label{lemma_zeta}
Suppose that condition H(a) holds. Then for any given $\lambda\in L^2(I;Y_1)$, there exists a unique $\zeta\in H^1(I;Y_1)\cap L^2(I;Y)$ such that
\begin{eqnarray}\label{equ_last_theo}
\langle \dot{\zeta}(t),\eta-\zeta(t) \rangle_{Y_1}+a(\zeta(t) ,\eta-\zeta(t) )\geq\langle\lambda(t),\eta-\zeta(t) \rangle_{Y_1},\quad\forall\eta\in K_Y,
\end{eqnarray}
with $\zeta(0)=\zeta_0\in K_Y$.  Moreover, if $\zeta_i$ is the unique solution to problem \eqref{equ_last_theo} for $\lambda_i\in L^2(I;Y_1)$ with $i=1,2$, then
\begin{eqnarray}
\left\|\zeta_{1}(t)-\zeta_{2}(t)\right\|_{Y_1}^{2} \leq d_{1} \int_{0}^{t}\left\|\lambda_{1}(s)-\lambda_{2}(s)\right\|_{Y_1}^{2} d s \quad \text { for a.e. } t \in(0, T)
\end{eqnarray}
with $d_1>0.$
\end{lemma}
\begin{lemma}\label{lemma_last}
 Suppose that conditions  H(A), H(B), H(C), H(j), H(F), H($\phi$), H(a) hold and $ m_C>\alpha_1$. Then Problem \ref{problem_last} has a unique solution $\zeta\in H^1(I;Y_1)\cap L^2(I;Y). $
\end{lemma}
\begin{proof}
For any given $\theta\in L^2(I;Y_1)$, it follows from Lemma \ref{lemma_aux_2} that Problem \ref{problem_aux_2} has a unique solution  $(u_\theta,w_\theta)\in C(I;K_V)\times C^1(I;W)$. Let $\phi_{\theta}(t):=\phi(t,u_{\theta}(t),\theta(t))$.  We claim that $\phi_{\theta}\in L^2(I;Y_1)$. Indeed, the condition H($\phi$)(a) derives the following inequality
 \begin{eqnarray}\label{equ_last1}
 \|\phi_\theta(t)\|^2_{Y_1}&\leq&2\|\phi(t,u_\theta(t),\theta(t))-\phi(t,0_V,0_{Y_1})\|^2_{Y_1}
 +2\|\phi(t,0_V,0_{Y_1})\|^2_{Y_1}\nonumber\\
 &\leq&2(L_{\phi}(\|u_\theta(t)\|_V+\|\theta\|_{Y_1}))^2+2\|\phi(t,0_V,0_{Y_1})\|^2_{Y_1}\nonumber\\
  &\leq&4L_{\phi}^2(\|u_\theta(t)\|^2_V+\|\theta\|^2_{Y_1}) +2\|\phi(t,0_V,0_{Y_1})\|^2_{Y_1}.
 \end{eqnarray}
Thus inequality \eqref{equ_last1} and condition H($\phi$)(b) show that
 \begin{eqnarray}
 \|\phi_\theta(t)\|^2_{L^2(I;Y_1)}\leq 4L^2_{\phi}T\|u_\theta(t)\|^2_{C(I:V)}+4L^2_{\phi}\|\theta\|^2_{L^2(I;Y_1)}+2\|\phi(t,0_V,0_{Y_1})\|^2_{L^2(I;Y_1)}
 \end{eqnarray}
 and so $\phi_{\theta}(t)\in L^2(I;Y_1)$. Taking $\lambda=\phi_\theta$ in Lemma \ref{lemma_zeta}, we know that there exists a unique $\zeta_\theta\in H^1(I;Y_1)\cap L^2(I;Y)$ such that inequality \eqref{equ_last_theo} holds.

Next we define an operator $\Lambda:L^2(I;Y_1)\rightarrow H^1(I;Y_1)\cap L^2(I;Y)$ by  $\Lambda(\theta)(t)=\zeta_\theta(t).$ Because $\left(Y, Y_1, Y^{*}\right)$ is a Gelfand triplet with dense and compact embeddings, one has $H^1(I;Y_1)\cap L^2(I;Y)\subset L^2(I;Y_1)$.
 The rest is to show that $\Lambda$ has a unique fixed point in $L^2(I;Y_1).$ To this end, let $\theta_i\in L^2(I;Y_1)$ with $i=1,\;2.$ Then Lemma \ref{lemma_zeta} yields
 \begin{eqnarray*}
\|\Lambda(\theta_1)(t)-\Lambda(\theta_2)(t)\|^2_{Y_1}&=&\|\zeta_{\theta_1}(t_1)-\zeta_{\theta_2}(t_2)\|^2_{Y_1}\nonumber\\
&\leq&d_{1} \int_{0}^{t}\left\|\phi(s,u_{\theta_1}(s),\theta_1(s))-\phi(s,u_{\theta_2}(s),\theta_2(s))\right\|_{Y_1}^{2} ds.
 \end{eqnarray*}
Using condition H($\phi$)(a) to the above inequality obtains
  \begin{eqnarray}\label{equ_lem_zeta1}
\|\Lambda(\theta_1)(t)-\Lambda(\theta_2)(t)\|^2_{Y_1}
&\leq&d_{1} \int_{0}^{t}\left\|\phi(s,u_{\theta_1}(s),\theta_1(s))-\phi(s,u_{\theta_2}(s),\theta_2(s))\right\|_{Y_1}^{2} ds\nonumber\\
&\leq&2 d_{1} L_{\phi}\int_{0}^{t}\left(\|u_{\theta_1}(s)-u_{\theta_2}(s)\|_{V}^{2} +\|\theta_1(s)-\theta_2(s)\|_{Y_1}^{2} \right)ds.
 \end{eqnarray}
 By Remark \ref{r3.2}, we can adapt the inequality \eqref{equ_lem_zeta1} as
 \begin{eqnarray}\label{equ_lem_zeta2}
\|\Lambda(\theta_1)(t)-\Lambda(\theta_2)(t)\|^2_{Y_1}
&\leq& 2d_{1} L_{\phi}\int_{0}^{t}\left(\int_{0}^{t}\|\dot{u}_{\theta_1}(s)-\dot{u}_{\theta_2}(s)\|_{V}^{2}ds +\|\theta_1(s)-\theta_2(s)\|_{Y_1}^{2} \right)ds\nonumber\\
&\leq&  c\|\theta_1-\theta_2\|^2_{L^2(I;Y_1)},
 \end{eqnarray}
where
$$c= 2d_{1} L_{\phi}\left(\left(( L_Fc_qc_rT^2+L_Fc_pc_rT)e^{(L_F+L_F c_p+L_Fc_qT)T}+c_r\right)^2T^2+1\right).$$
Thus inequality \eqref{equ_lem_zeta2} shows that
 \begin{eqnarray}\label{equ_lem_zeta3}
  \|\Lambda(\theta_1)(t)-\Lambda(\theta_2)(t)\|_{L^(I;Y_1)}\leq  \sqrt c\int_{0}^{t}\|\theta_1-\theta_2\|_{L^2(I;Y_1)}.
 \end{eqnarray}
Now it follows from Lemma \ref{lemma_hisdep} that there exists a unique element $\theta^*$ in $L^2(I;Y_1)$ such that $\theta^*$ is the unique fixed point of $\Lambda$, i.e., $\theta^*=\Lambda(\theta^*)=\zeta_{\theta^*}\in H^1(I;Y_1)\cap L^2(I;Y)$. Thus,  $\theta^*$ is a unique solution of Problem \ref{problem_last}. This completes the proof.
\end{proof}

Finally, we can give the proof of Theorem \ref{t3.1} as follows.
\begin{proof}
By Remark \ref{remark_solu} and Lemma \ref{lemma_last}, it is easy to see that $(\zeta,u_\zeta,w_\zeta)\in  (H^1(I;Y_1)\cap L^2(I;Y))\times C^1(I;K_Y)\times C^1(I;W)$ is a unique solution of DQVIs \eqref{1.1}.  This completes the proof of Theorem \ref{t3.1}.
\end{proof}

\section{An application to a contact problem for viscoelastic materials with long memory, damage and wear}
\setcounter{equation}{0}
In this section, we will use the results presented in Section 3 to solve a frictional contact problem of viscoelastic materials with long memory, damage and wear. Assume that $\Omega$ is a bounded and open domain in $\mathbb{R}^{d}(d=2,3)$ and its boundary $\Gamma:=\partial \Omega$ is Lipschitz continuous. For any $\bm{x}\in\Omega$,  we denote it as $\bm{x}=(x_1,x_2,\cdots,x_d)$. Let $\bm{\nu}$ denote the unit outward normal vector defined $a.e.$ on $\Gamma$.  The space of second order symmetric tensors on $\mathbb{R}^d$ is denoted by $\mathbb{S}^d$. Moreover, assume that $\mathbb{R}^{d}$ and $\mathbb{S}^{d}$ are equipped with the following inner products and norms
\begin{align*}
&\bm{u\cdot v}=u_{i}v_{i},\qquad\;\; \|\bm{v}\|_{\mathbb{R}^{d}} =(\bm{v\cdot v})^{\frac{1}{2}}, \qquad\forall   \bm{u}  ,\bm{v}  \in \mathbb{R}^{d},         \\
& \bm{\sigma\cdot \tau}=\sigma_{ij}\tau_{ij},\qquad \|\bm{\tau}\|_{\mathbb{S}^{d}} =(\bm{\tau\cdot \tau})^{\frac{1}{2}}, \qquad\forall   \bm{\sigma}  ,\bm{\tau} \in \mathbb{S}^{d},
\end{align*}
where repeated indices represent summation convention.

We use notations $\bm{u}=(u_i)$, $\bm{\sigma}=(\sigma_{ij})$ and $\bm{\varepsilon(u)}=(\varepsilon_{ij}(\bm{u}))=(\frac{1}{2}(u_{i,j}+u_{j,i})), i,j=1,2,\cdots,d$ to denote the displacement vector, the stress tensor and the linearized strain tensor, respectively, where $u_{i,j}\equiv \frac{\partial u_i}{\partial x_j}$. Here and below, the spatial derivative is defined in the sense of distribution. The normal and tangential components of stress field $\bm{\sigma}$ on $\Gamma$ are denoted by $\sigma_\nu=(\bm{\sigma\nu})\bm{\cdot}\bm{\nu}$ and $\bm{\sigma_\tau}=\bm{\sigma}\bm{\nu}-\sigma_\nu\bm{\nu}$, respectively. The normal and tangential components of the displacement field on $\Gamma$ are denoted by $u_\nu=\bm{u}\bm{\cdot}\bm{\nu}$ and $\bm{u_\tau}=\bm{u}-u_\nu\bm{\nu}$, respectively.
We are interested in the evolution of the body on the time interval $I:=[0,T]$, where $T>0$ is a constant. The time partial derivative for a function $f(\bm{x},t)$ is denoted by $\dot{f}(\bm{x},t)$. For the sake of simplicity, we usually omit the variable $\bm{x}$ in functions.

Now we are in the position to describe the physical background of viscoelastic materials' quasistatic frictional contact problems with damage, wear and long memory. We consider a viscoelastic body which occupies $\Omega$. The boundary $\partial \Omega$ is composed of three disjoint measurable parts $\Gamma_{1}$, $\Gamma_{2}$ and $\Gamma_{3}$ with $\mbox{meas}(\Gamma_{1})>0$.  Volume forces with density $\bm{f}_{0}$ act on $\Omega$ and surface tractions with density $\bm{f}_{2}$ act on $\Gamma_{2}$. On $\Gamma_1$, we suppose that the body is clamped, i.e., the displacement field disappears there.  The contact model is established by a normal compliance condition on $\Gamma_3$ with sliding Coulomb's law of dry friction and wear. The foundation is made of a  perfectly rigid material. The velocity of the foundation is denoted by $\bm{v}^*(t)\neq\bm{0}$.
Let $\boldsymbol{n}^{*}$ be the direction of the foundation's velocity and $\alpha$ be wear rate which are given respectively by
$$\boldsymbol{n}^{*}(t)=-\boldsymbol{v}^{*}(t) /\left\|\boldsymbol{v}^{*}(t)\right\|, \quad \alpha(t)=k\left\|\boldsymbol{v}^{*}(t)\right\|,$$
where $k$ stands for the wear coefficient. We use the following abbreviations to simplify the notations $\mathcal{Q}=\Omega\times I$, $\Sigma=\Gamma\times I$, $\Sigma_i=\Gamma_i\times I$, $i=1,2,3.$

Thus, we can present the formulation of the contact problem as follows.
\begin{problem}\label{mechanic_problem1}
Find a displacement field $\bm u:\mathcal{Q}\rightarrow\mathbb{R}^d$, a stress field $\bm{\sigma}:\mathcal{Q}\rightarrow\mathbb{S}^d$, a damage field $\zeta:\Sigma_3\rightarrow[0,1]$ and a wear function $w:\Sigma_3\rightarrow\mathbb{R}$ such that
\begin{align}
\bm{\sigma}(t)= \mathcal{A}\left(t, \bm{\varepsilon}(\bm{u}(t))\right)+\int_{0}^{t}\mathcal{B}\left(t-s, \bm{\varepsilon}(\bm{u}(s)), \zeta(s)\right) ds+ \mathcal{C}(t, \bm{\varepsilon}(\dot{\bm{u}}(t)))\quad&\quad \mbox{in}\quad \;\mathcal{Q},\label{eq_mep_1}\\
\dot{\zeta}-\kappa \Delta \zeta+\partial I_{[0,1]}(\zeta) \ni \phi(\bm{\varepsilon}(\bm{u}(t)), \zeta)\quad&\quad \mbox{in}\quad \;\mathcal{Q},\label{eq_dam_1}\\
\frac{\partial \zeta}{\partial \nu}=0\quad&\quad \mbox{on}\quad \;\Sigma,\label{eq_dam_2}\\
-\mbox{Div}\bm{\sigma}(t)=\bm{f}_{0}(t)\quad&\quad \mbox{in}\quad \;\mathcal{Q},\label{eq_mep_2}\\
\bm{u}(t)=\bm{0}\quad&\quad \mbox{on}\quad \Sigma_{1},\label{eq_mep_3}\\
\bm{\sigma}(t)\bm{\nu}=\bm{f}_{2}(t)\quad&\quad \mbox{on}\quad \Sigma_{2},\label{eq_mep_4}\\
\dot{u}_\nu(t)\leq g,\; \sigma_{\nu}(t)+p(w(t),\dot{u}_\nu(t))\leq 0
\quad&\quad\mbox{on}\quad \Sigma_{3},\label{eq_mep_7}\\
(\dot{u}_{\nu}(t)-g)(\sigma_{\nu}(t)+p(w(t),\dot{u}_\nu(t)))= 0 \quad&\quad\mbox{on}\quad \Sigma_{3},\label{eq_mep_8}\\
-\bm{\sigma}_{\tau}(t)= \mu p(w(t),\dot{u}_\nu)\cdot\bm{n^*}(t) \quad&\quad \mbox{on}\quad \;\Sigma_{3},\label{eq_mep_9}\\
\dot{w}(t)=\alpha(t)p(w(t),\dot{u}_\nu(t))\quad&\quad \mbox{on}\quad \;\Sigma_{3},\label{eq_mep_10}\\
\bm{u}(0)=\bm{u}_0,w(0)=0,\zeta(0)=\zeta_0\in(0,1)\quad&\quad \mbox{on}\quad \;\Gamma_{3}\label{eq_mep_12}.
\end{align}
\end{problem}
\begin{remark}\label{rem_subdiff}
The conditions \eqref{eq_mep_7} and \eqref{eq_mep_8} can be rewritten as
$$-\sigma_\nu(t)\in \partial I_K(\dot{u}_\nu)+p(w(t),\dot{u}_\nu),$$
where $K:=\{s\in \mathbb{R}\:|\:s\leq g\}$, $I_K$ is the indication function of $K$  and $\partial I_K$ is the convex subdifferential of $I_K$. From the definition of convex subdifferential, one has
$$(-\sigma_\nu(t)-p(\dot{u}_\nu(t)-w(t)))\cdot(s-\dot{u}_\nu)\leq I_K(s)-I_K(\dot{u}_\nu)=0,\quad \forall s\in K.$$
\end{remark}
Equation \eqref{eq_mep_1} stands for the viscoelastic constitutive law with damage effect and long memory, where $\zeta$ is the damage function, $\mathcal{A}$, $\mathcal{B}$ and $\mathcal{C}$ are elasticity, relaxation and viscosity operators, respectively.
The evolution law of damage function can be modeled from a parabolic differential inclusion \eqref{eq_dam_1}, where $\phi$ is the mechanical source of damage which depends on the damage itself and the strain, $I_{[0,1]}$ is the indication function of the interval $[0,1]$, $\partial I_{[0,1]}$ is the convex subdifferential of $I_{[0,1]}$ and $\kappa >0$ is a constant mircrocrack diffusion coefficient. We select the following damage source function
$$
\phi(\varepsilon(\boldsymbol{u}), \zeta) \equiv \phi_{\mathrm{Fr}}(\|\varepsilon(\boldsymbol{u})\|, \zeta)=\lambda_{D}\left(\frac{1-\zeta}{\zeta}\right)-\frac{1}{2} \lambda_{E}\|\varepsilon(\boldsymbol{u})\|^{2}+\lambda_{w},
$$
where $\lambda_{D}, \lambda_{E},$ and $\lambda_{w}$ are three positive numbers (\cite{Fr1995,Fr1996}).
Equation \eqref{eq_dam_2} represents a homogeneous Neumann condition of $\zeta$ on $\Sigma$ and equation \eqref{eq_mep_2} describes the  equilibrium equation. Equations \eqref{eq_mep_3} and \eqref{eq_mep_4} denote the clamped boundary condition and the tractive boundary conditions, respectively.
Equations \eqref{eq_mep_7} and \eqref{eq_mep_8} are called
normal damped condition with unilateral constraint, which have been studied in \cite{Eck2010} without wear. In the previous literature, the normal compliance function $p(\dot{u}_\nu)$ was used to model the
lubricated contact.
 In this paper, we assume that the lubrication is incomplete and the normal compliance depends on wear, i.e., $p$ has the form of $p(w,\dot{u}_\nu)$ with $w$ being the wear.
Equation \eqref{eq_mep_9} gives the functional relationship between the wear coefficient $k$ and the friction coefficient which denoted $\mu$. Here we assume that $\bm{v^*}(t)$ is much larger than the tangential body velocity $\bm{\dot{u}_\tau}(t)$. Equation \eqref{eq_mep_10} is related to the development of the wear function which can be obtained from the Archard's law with $\alpha(t)=k\|\bm{v^*}(t)\|$. For more details concerned with equations \eqref{eq_mep_7}-\eqref{eq_mep_10}, we can refer to the literature \cite{SofoneaF2016}.
Finally, equation \eqref{eq_mep_12} gives the initial conditions for the displacement field, the wear function and the damage field.

In order to get the variational formulation of Problem \ref{mechanic_problem1}, we need the following spaces. Let $W^{ k,p}(\Omega;\mathbb{R}^d)$ denote the Sobolev space of all functions. On $\Omega$, the weak derivatives of functions with order less than or equal to $k$ are $p$-integrable. Specially, let $H^1:=W^{ 1,2}(\Omega;\mathbb{R}^d)$ and $H=L^2(\Omega;\mathbb{R}^d).$
Let $V=\{\bm{v}\in H^1|\bm{v}=0\; \mbox{a.e. on} \;\Gamma_3\}$ endowed with the norm $$\|\bm{u}\|_V:=\|\bm{u}\|_{H^1}=\|\bm{u}\|_{L^2(\Omega;\mathbb{R}^d)}+\|\triangledown\bm{u}\|_{L^2(\Omega;\mathbb{R}^{d\times d})},$$
where $\nabla \bm{u}=(\frac{\partial u_{i}}{\partial x_j})$ for $i,j=1,\cdots,d$  with  $\bm{u}\in H^1.$
Let $\mbox{Div}\bm{\sigma}=(\sigma_{ij,j})=(\frac{\partial\sigma_{ij}}{\partial x_j})$ with $\bm{\sigma}\in W^{1,2}(\Omega;\mathbb{S}^d)$. Then we have the following Green formula:
$$(Div\bm{\sigma},\bm{v})_H+(\bm{\sigma},\bm{\varepsilon(u)})_{L^2(\Omega;\mathbb{S}^d)}=\int_{\Gamma}\bm{\sigma\nu}\cdot\bm{v} d_{\Gamma},$$
where $(\bm{\sigma},\bm{\varepsilon(u)})_{L^2(\Omega;\mathbb{S}^d)}=\int_{\Omega}\bm{\sigma}\bm{\cdot}\bm{\varepsilon(u)}d\Omega$. From the assumption of $\mbox{meas}(\Gamma_1) > 0$, we know that Korn's inequality holds, i.e.,
$$\|\bm{\varepsilon}(\bm{v})\|_{L^2(\Omega;\mathbb{S}^d)}\geq c\|\bm{v}\|_V,\quad \forall \bm{v}\in V.$$
Here and then, $c$  represents a positive constant which may change from line to line. We can endow with inner product on $ V$ as
$$\langle\bm{u},\bm{v}\rangle_{ V}=\langle\bm{\varepsilon}(\bm{u}),\bm{\varepsilon}(\bm{v})\rangle_{L^2(\Omega;\mathbb{S}^d)}.$$

Let $Y=H^1(\Omega;\mathbb{R})$ and $Y_1=L^2(\Omega;\mathbb{R})$ endowed with the inner products
$$(w,z)_{Y_1}=\int_{\Omega}wzd_{\Omega}, \quad (w,z)_Y=(w,z)_{Y_1}+\int_{\Omega}\nabla w\cdot \nabla zd_{\Omega}.$$
Denote $K_V=\{\bm{v}\in V\:|\:v_\nu\leq g\;\mbox{a.e. on}\;\Gamma_3\}$ and $K_Y=\{u\in Y\:|\:0\leq u\leq 1 \; \mbox{a.e. in}\;\Omega\}$. Then $K_V$ and $K_Y$ are both convex.
Suppose that $\{\bm{u},\bm{\sigma},\zeta,w\}$ are sufficiently smooth functions solving \eqref{eq_mep_1}-\eqref{eq_mep_12} with $t \in I$.
Then, we can employ the following method to derive the variational formulation of Problem \ref{mechanic_problem1}.
Firstly, we use the Green formula and the equation \eqref{eq_mep_1} to obtain
\begin{equation}\label{eq_PV_1}
\langle\boldsymbol{\sigma}(t), \bm{\varepsilon}(\bm{v})-\bm{\varepsilon}(\bm{\bm { \dot{u} }}(t))\rangle_{L^2(\Omega;\mathbb{S}^d)}=\left\langle\bm{f}_{0}(t), \bm{v}-{\bm{\dot{u}}}(t)\right\rangle_{L^2(\Omega;\mathbb{R}^d)}+\int_{\Gamma} \bm{\sigma}(t) \bm{\nu} \cdot(\bm{v}-{\bm{\dot{u}}}(t)) d \Gamma.
\end{equation}
Considering the boundary conditions \eqref{eq_mep_3}, \eqref{eq_mep_4}, \eqref{eq_mep_9} and the formula
$$
\boldsymbol{\sigma}(t) \boldsymbol{\nu} \cdot \boldsymbol{v}=\sigma_{\nu}(t) v_{\nu}+\boldsymbol{\sigma}_{\tau}(t) \cdot \boldsymbol{v}_{\tau},
$$
the equation \eqref{eq_PV_1} can be transformed as
\begin{eqnarray}\label{eq_PV_2}
&&\langle\boldsymbol{\sigma}(t), \bm{\varepsilon}(\bm{v})-\bm{\varepsilon}(\bm{\dot{u}}(t))\rangle_{L^2(\Omega;\mathbb{S}^d)}\nonumber\\
&=&\left\langle\bm{f}_{0}(t), \bm{v}-{\bm{\dot{u}}}(t)\right\rangle_{L^2(\Omega;\mathbb{R}^d)}+\int_{\Gamma_2} \bm{f_2}(t) \cdot(\bm{v}-{\bm{\dot{u}}}(t)) d \Gamma \nonumber-\int_{\Gamma_3}\mu p(w(t),\dot{u}_\nu(t)) \bm{n^*}(t) \cdot(\bm{v}_\tau-{\bm{\dot{u}}_\tau}(t)) d \Gamma\nonumber\\
&&\mbox{}+\int_{\Gamma_3} {\sigma}_\nu(t) \cdot(v_\nu-{\dot{u}_\nu}(t)) d \Gamma.
\end{eqnarray}
From Remark \ref{rem_subdiff} and $\bm{v}\in K_V$, one has
$$\int_{\Gamma_3} {\sigma}_\nu(t) \cdot(v_\nu-{\dot{u}_\nu}(t)) d \Gamma\geq\int_{\Gamma_3}-p(w(t),\dot{u}_\nu(t))(v_\nu-{\dot{u}_\nu}(t)) d \Gamma.$$
Now \eqref{eq_PV_2} becomes
\begin{eqnarray}\label{eq_PV_3}
&&\langle\boldsymbol{\sigma}(t), \bm{\varepsilon}(\bm{v})-\bm{\varepsilon}(\bm{\bm { \dot{u} }}(t))\rangle_{L^2(\Omega;\mathbb{S}^d)}+\int_{\Gamma_3}\mu p(w(t),\dot{u}_\nu(t)) \bm{n^*}(t) \cdot(\bm{v}_\tau-{\bm{\dot{u}}_\tau}(t)) d \Gamma\nonumber \\
&&\mbox{}+\int_{\Gamma_3}p(w(t),\dot{u}_\nu(t))(v_\nu-{\dot{u}_\nu}(t))d \Gamma\nonumber\\
&\geq&\left\langle\bm{f}_{0}(t), \bm{v}-{\bm{\dot{u}}}(t)\right\rangle_{L^2(\Omega;\mathbb{R}^d)}+\int_{\Gamma_2} \bm{f_2}(t) \cdot(\bm{v}-{\bm{\dot{u}}}(t)) d \Gamma.
\end{eqnarray}
Let
$$
j(w,\gamma\bm{\dot{u}},\gamma\bm{v}):={\int_{\Gamma_3}\mu p(w(t),\dot{u}_\nu(t)) \bm{n^*}(t) \cdot\bm{v}_\tau d \Gamma+\int_{\Gamma_3}p(w(t),\dot{u}_\nu(t))v_\nu d \Gamma.}
$$
Then it follows from \eqref{eq_PV_3} that
\begin{eqnarray}\label{eq_PV_4}
&&\langle\boldsymbol{\sigma}(t), \bm{\varepsilon}(\bm{v})-\bm{\varepsilon}(\bm{\bm { \dot{u} }}(t))\rangle_{L^2(\Omega;\mathbb{S}^d)}
+j(w,\gamma\bm{\dot{u}},\gamma\bm v)-j(w,\gamma\bm{\dot{u}},\gamma\bm{\dot{u}})\nonumber\\
&\geq& \left\langle \bm{f_2}(t) ,\gamma\bm{v}-\gamma{\bm{\dot{u}}}(t)\right\rangle_{L^2(\Gamma_3;\mathbb{R}^d)}  +\left\langle\bm{f}_{0}(t),\bm{v}-{\bm{\dot{u}}}(t)\right\rangle_{L^2(\Omega;\mathbb{R}^d)},
\end{eqnarray}
where $\gamma :V\rightarrow L^2(\Gamma_3;\mathbb{R}^d)$ is the trace operator. Applying Korn's inequality, the Riesz representation theorem and the trace theorem, we know that there exists $\bm{f}:I\rightarrow V^*$ such that
$$\langle\bm f(t),\bm v-\bm {\dot{u}}\rangle_{V^*\times V}=\left\langle \bm{f_2}(t) ,\gamma\bm{v}-\gamma{\bm{\dot{u}}}(t)\right\rangle_{L^2(\Gamma_3;\mathbb{R}^d)}  +\left\langle\bm{f}_{0}(t),\bm{v}-{\bm{\dot{u}}}(t)\right\rangle_{L^2(\Omega;\mathbb{R}^d)}.$$
Similarly, from the integration by parts and the definition of convex subdifferential of $I_{[0,1]}$, it follows that
\begin{eqnarray}\label{eq_PV_5}
\langle\phi(t, \bm{\varepsilon}(\boldsymbol{u}(t)), \zeta(t))-\dot{\zeta}(t), \eta-\zeta(t)\rangle_{L^{2}(\Omega;\mathbb{R})}\leq a(\zeta,\eta-\zeta),
\end{eqnarray}
where $a(\zeta,\eta)=\kappa\int_\Omega\nabla\zeta\cdot\nabla\eta dx$ for all $\zeta,\eta\in Y.$

Now, by integrating these relations and inequalities, the variational formulation of Problem \ref{mechanic_problem1} can be obtained as follows.
\begin{problem}\label{problem_1}
Find $\bm{u}:I\rightarrow K_V$, $\zeta:I\rightarrow K_Y$ and $w:I\rightarrow L^2(\Gamma_3;\mathbb{R})$ such that, for a.e. $t \in I$,
\begin{eqnarray}
\bm{\sigma}(t)= \mathcal{A}(t, \bm{\varepsilon}(\bm{u}(t)))+\int_{0}^{t}\mathcal{B}(t-s, \bm{\varepsilon}(\bm{u}(s)), \zeta(s)) ds+ \mathcal{C}(t, \bm{\varepsilon}(\bm{\dot{u}}(t))) \quad \text { in } \Omega,\quad\\
\langle\boldsymbol{\sigma}(t), \bm{\varepsilon}(\bm{v})-\bm{\varepsilon}(\bm{\bm {\dot{u}(t) }})\rangle_{L^2(\Omega;\mathbb{S}^d)}
+j(w,\gamma\bm{\dot{u}},\gamma\bm v)-j(w,\gamma\bm{\dot{u}},\gamma\bm{\dot{u}})
\geq\left\langle\bm{f}(t), \bm{v}-{\bm{\dot{u}}}(t)\right\rangle_{V^*\times V},\; \forall \bm v\in K_V,\quad\label{eq_u_continuous}\\
\langle(\dot{\zeta}(t), \eta-\zeta(t)\rangle_{Y_1}
+a(\zeta,\eta-\zeta)\geq \langle\phi(t, \bm{\varepsilon}(\boldsymbol{u}(t)),\zeta(t)),\eta-\zeta(t)\rangle_{Y_1}
,\; \forall \eta\in K_Y,\quad\\
\dot{w}(t)=\alpha(t)p(\dot{u}_\nu(t)-w(t))\quad\mbox{ in } \Gamma_3,\quad\label{eq_w_continuous}\\
\bm{u}(0)=\bm{u}_0,w(0)=0,\;\zeta(0)=\zeta_0\in(0,1).\quad
\end{eqnarray}
\end{problem}

To solve Problem \ref{problem_1}, we need the following assumptions.

$H(1):$ The elasticity operator
$\mathcal{A}: \Omega  \times I \times\mathbb{S}^{d} \rightarrow \mathbb{S}^{d}$ satisfies
\begin{align}\label{assume:A}
\begin{cases}
(a)\; \mathcal{A}(\cdot, t,\bm{\varepsilon}) \;\mbox{is measurable on}\; \Omega, \mbox{ for all } ( t,\bm{\varepsilon}) \in I\times \mathbb{S}^{d};\\
(b)\; \mathcal{A}(\bm{x},\cdot, \cdot) \;\mbox{is continuous on }I \times\mathbb{S}^{d} \mbox{ for a.e. }\; \bm{x} \in \Omega;\\
(c)\; \mathcal{A}(\bm{x},t,  \cdot) \mbox{ is Lipschitz continous with}\; m_{\mathcal{A}}>0\; \mbox{for all }\; t \in I,\; i.e.,\\
\quad\left\|\mathcal{A}\left(\bm{x}, t, \bm{\varepsilon}_{1}\right)-\mathcal{A}\left(\bm{x}, t, \bm{\varepsilon}_{2}\right)\right\|\leq L_{\mathcal{A}}\left\|\bm{\varepsilon}_{1}-\bm{\varepsilon}_{2}\right\|,\;\forall \bm{\varepsilon}_{1}, \bm{\varepsilon}_{2} \in \mathbb{S}^{d}, \mbox{ a.e. } \bm{x} \in \Omega;\\
\end{cases}
\end{align}

$H(2):$ The relaxation operator $\mathcal{B}: \Omega \times I \times \mathbb{S}^{d} \times \mathbb{R} \rightarrow \mathbb{S}^{d}$ satisfies
\begin{align}\label{assume:B}
\begin{cases}
(a) \mathcal{B}(\cdot,t, \bm{\varepsilon},\zeta) \;\mbox{is measurable on}\; \Omega, \mbox{ for all } \bm{\varepsilon} \in \mathbb{S}^{d},\; t\in I \mbox{ and } \zeta\in\mathbb{R};\\
(b) \mathcal{B}(\bm{x}, \cdot,\bm{\varepsilon},\zeta) \mbox{ is continuous on } I\\
 \quad\mbox{ for a.e. } \bm{x} \in \Omega \mbox{ and all } (\bm{\varepsilon},\zeta) \in \mathbb{S}^{d}\times \mathbb{R};\\
(c) \mathcal{B}(\bm{x}, t,\cdot,\cdot) \;\mbox{is Lipschitz continuous with}\; L_\mathcal{B}> 0 \mbox{ for all } t\in I \mbox{ and a.e.}\; \bm{x} \in \Omega, i.e.,\\
\quad\|\mathcal{B}(\bm{x},t,\bm{\varepsilon}_{1},\zeta_1)-\mathcal{B}(\bm{x},t,\bm{\varepsilon}_{2},\zeta_2)\| \leq L_{\mathcal{B}}(\|\bm{\varepsilon}_{1}-\bm{\varepsilon}_{2}\|+|\zeta_1-\zeta_2|),\\
\quad\forall \bm{\varepsilon}_{1}, \bm{\varepsilon}_{2} \in \mathbb{S}^{d},\;
\zeta_1,\zeta_2\in\mathbb{R}, \mbox{ a.e. } \bm{x} \in \Omega;\\
(d)\mbox{For all } (t,\bm{\varepsilon},\zeta)\in I\times \mathbb{S}^{d}\times \mathbb{R}\mbox{ and a.e.}\; \bm{x} \in \Omega,\;\mbox{there exixsts a function }\\
\quad\rho_\mathcal{B}\in L^2(I;\mathbb{R}^+) \mbox{ such that } \|\mathcal{B}(\bm{x},t,\bm{\varepsilon},\zeta)\|\leq\rho_\mathcal{B}(t)(|\zeta|+\|\varepsilon\|).
\end{cases}
\end{align}

$H(3):$ The viscosity operator
$\mathcal{C}: \Omega  \times I \times\mathbb{S}^{d} \rightarrow \mathbb{S}^{d}$ satisfies
\begin{align}\label{assume:C}
\begin{cases}
(a)\; \mathcal{C}(\cdot, t,\bm{\varepsilon}) \;\mbox{is measurable on}\; \Omega \mbox{ for all } (t,\bm{\varepsilon}) \in I\times \mathbb{S}^{d};\\
(b)\; \mathcal{C}(\bm{x},\cdot, \cdot) \;\mbox{is continuous on }I \times\mathbb{S}^{d} \mbox{ for a.e. }\; \bm{x} \in \Omega;\\
(c)\; \mathcal{C}(\bm{x},t,  \cdot) \mbox{ is strongly monotone with}\; m_{\mathcal{C}}>0\; \mbox{for all }\; t \in I,\; i.e.,\\
\quad\left(\mathcal{C}\left(\bm{x},t,\bm{\varepsilon}_{1}\right)-\mathcal{C}\left(\bm{x},t,\bm{\varepsilon}_{2}\right)\right)\cdot\left(\bm{\varepsilon}_{1}-\bm{\varepsilon}_{2}\right) \geq m_{\mathcal{C}}\left\|\bm{\varepsilon}_{1}-\bm{\varepsilon}_{2}\right\|^2,\;\forall \bm{\varepsilon}_{1}, \bm{\varepsilon}_{2} \in \mathbb{S}^{d}, \mbox{ a.e. } \bm{x} \in \Omega;\\
(d)\; \mathcal{C}(\bm{x},t,  \cdot) \mbox{ is Lipschitz continous with}\; L_{\mathcal{C}1}>0\; \mbox{for all }\; t \in I,\; i.e.,\\
\quad\left\|\mathcal{C}\left(\bm{x}, t, \bm{\varepsilon}_{1}\right)-\mathcal{C}\left(\bm{x}, t, \bm{\varepsilon}_{2}\right)\right\|\leq L_{\mathcal{C}1}\left\|\bm{\varepsilon}_{1}-\bm{\varepsilon}_{2}\right\|,\;\forall \bm{\varepsilon}_{1}, \bm{\varepsilon}_{2} \in \mathbb{S}^{d}, \mbox{ a.e. } \bm{x} \in \Omega;\\
(e)\; \mathcal{C}(\bm{x},\cdot, \bm{\varepsilon}) \mbox{ is Lipschitz continous with}\; L_{\mathcal{C}2}>0\; \mbox{for all }\; \bm{\varepsilon}\in \mathbb{S}^{d},\; i.e.,\\
\quad\left\|\mathcal{C}\left(\bm{x}, t_1, \bm{\varepsilon}\right)-\mathcal{C}\left(\bm{x}, t_2, \bm{\varepsilon}\right)\right\|\leq L_{\mathcal{C}2}\left\|t_{1}-t_{2}\right\|,\;\forall t_{1}, t_{2} \in I, \mbox{ a.e. } \bm{x} \in \Omega;\\
\end{cases}
\end{align}

$H(4):$ The normal compliance function  $p: \Gamma_3 \times\mathbb{R} \times \mathbb{R} \rightarrow \mathbb{R}^+$ satisfies
\begin{align}\label{assume:p}
\begin{cases}
(a)\; p(\cdot, w,u) \;\mbox{is measurable on}\; \Gamma_3  \mbox{ for all } w,\,u \in \mathbb{R};\\
(b)\; p(\bm{x}, \cdot,\cdot) \;\mbox{is Lipschitz continuous with}\; \tilde{L}_p> 0 \mbox{, a.e.}\; \bm{x} \in \Omega, i.e.,\\
\quad|p(\bm{x},w_1,u_1)-p(\bm{x}, w_2,u_2)| \leq \tilde{L}_p(|w_1-w_2|+|u_1-u_2|),\;\forall w_1,w_2,u_1,u_2 \in \mathbb{R};\\
(c)\; p(\bm{x}, 0,0)=0 \mbox{ for a.e. } \bm{x}\in \Omega \mbox{ and all } r\leq0.
\end{cases}
\end{align}

$H(5):$ The damage source function  $\phi: \Omega \times I\times \mathbb{S}^{d} \times \mathbb{R} \rightarrow \mathbb{R}$ satisfies
\begin{align}\label{assume:phi}
\begin{cases}
(a) \phi(\cdot, t,\bm{\varepsilon},\zeta) \;\mbox{is measurable on}\; \Omega,\; \mbox{for all }t\in I,\; \bm{\varepsilon} \in \mathbb{S}^{d}\mbox{ and}\;\zeta\in\mathbb{R};\\
(b) \phi(\bm{x}, t,\cdot,\cdot) \;\mbox{is Lipschitz continuous with}\; L_\phi> 0 \mbox{ for all }t\in I \mbox{ a.e.}\; \bm{x} \in \Omega, i.e.,\\
\quad\|\phi(\bm{x},\boldsymbol{\varepsilon}_{1},\zeta_1)-\phi(\bm{x}, \boldsymbol{\varepsilon}_{2},\zeta_2)\| \leq \tilde{L}_\phi(\|\bm{\varepsilon}_{1}-\boldsymbol{\varepsilon}_{2}\|+|\zeta_1-\zeta_2|)\\
\quad\forall \bm{\varepsilon}_{1}, \bm{\varepsilon}_{2} \in \mathbb{S}^{d},\; \zeta_1,\zeta_2\in\mathbb{R}, \mbox{ a.e. } \bm{x} \in \Omega;\\
(c) \phi(\bm{x},t, \bm{0}_{\mathbb{S}^{d}},0_{\mathbb{R}})\in L^2(I;L^2(\Omega;\mathbb{R})) .
\end{cases}
\end{align}

Moreover, suppose the volume forces and surface tractions satisfy
\begin{equation}\label{assume:f}
\bm{f}_{0} \in C(I;H), \quad \bm{f}_{2} \in C\left(I; L^{2}\left(\Gamma_{2} ; \mathbb{R}^{d}\right)\right).
\end{equation}

To make the sliding condition (\ref{eq_mep_9}) and wear condition (\ref{eq_mep_10}) reasonable, the coefficient of friction, wear coefficient and velocity of the foundation are, respectively, assumed to be
\begin{equation}\label{assume:mu}
\mu\in L^\infty(\Gamma_3;\mathbb{R}),\;\mu(\bm{x})\geq0\;a.e. \;\bm{x}\in \Gamma_3,
\end{equation}
\begin{equation}\label{assume:wear}
k\in L^\infty(\Gamma_3;\mathbb{R}),\;k(\bm{x})\geq0\;a.e. \;\bm{x}\in \Gamma_3,
\end{equation}
and
\begin{equation}\label{assume:v}
\left\{\begin{array}{l}
\boldsymbol{v}^{*} \in C\left(I ; \mathbb{R}^{3}\right) \text { and there exist } v_{1}, v_{2}>0 \text { such that } \\
v_{1} \leq\left\|\boldsymbol{v}^{*}(t)\right\| \leq v_{2},  \;\forall t \in I.
\end{array}\right.
\end{equation}

Clearly, the conditions \eqref{assume:wear} and \eqref{assume:v} suggest that
\begin{equation}\label{assume:n}
\boldsymbol{n}^{*} \in C\left(I ; \mathbb{R}^{3}\right),\;\alpha\in C(I;L^\infty(\Gamma_3;\mathbb{R})).
\end{equation}
Moreover, $g\geq0$ yields that $\bm{0}_V\in K_V$.

For Banach spaces $V=\{\bm{u}\in H^1(\Omega;\mathbb{R}^d)\:|\:\bm{u}=0 \mbox{ on }\Gamma_3\}$ and $Y= H^1(\Omega;\mathbb{R})$, Hilbert spaces $H=L^2(\Omega;\mathbb{R}^d)$ and $Y_1= L^2(\Omega;\mathbb{R})$, by the basic theory of Sobolev spaces, we know that $(V,H,V^*)$ and $(Y,Y_1,Y^*)$ form two Gelfond triplets. Let $X=L^2(\Gamma_3;\mathbb{R}^d) $ and $W=L^2(\Gamma_3;\mathbb{R})$.

Now we give a unique solvability result for Problem \ref{problem_1} as follows.
\begin{theorem}\label{theorem_exist}
Let assumptions \eqref{assume:A}-\eqref{assume:n} hold.  If
$$m_{\mathcal{C}}>\|\mu\|_{L^\infty(\Gamma_3;\mathbb{R})}\tilde{L}(p)+\tilde{L}(p),$$
then Problem \ref{problem_1} has a unique solution $(\zeta,\bm{u}_\zeta,w_\zeta)\in  (H^1(I;Y_1)\cap L^2(I;Y))\times C(I;K_V)\times C^1(I;W)$.
\end{theorem}
\begin{proof} For any $t\in I$, define operators $A(t,\cdot):V\rightarrow V^*$, $B(t,\cdot,\cdot):V\times Y\rightarrow V^*$, $C(t,\cdot):V\rightarrow V^*$, $F(t,\cdot,\cdot):W\times V\rightarrow W$, $\phi(t,\cdot,\cdot):V\times Y\rightarrow Y^*$, a functional $j(\cdot,\cdot,\cdot):W\times X\times X\rightarrow \mathbb{R}$ and a symmetric bilinear form $a(\cdot,\cdot):Y\times Y\rightarrow \mathbb{R}$ by setting
\begin{align*}
\begin{cases}
\langle A(t,\bm{u}),\bm{v}\rangle_{V^*\times V}=\int_\Omega \mathcal{A}(\bm{x},t,\bm{\varepsilon}(\bm{u}))\bm{\cdot}\bm{\varepsilon}(\bm{v})d\Omega,\\
\langle B(t,\bm{u},\zeta),\bm{v}\rangle_{V^*\times V}=\int_\Omega \mathcal{B}(\bm{x},t,\bm{\varepsilon}(\bm{u}),\zeta)\bm{\cdot}\bm{\varepsilon}(\bm{v})d\Omega,\\
\langle C(t,\bm{u}),\bm{v}\rangle_{V^*\times V}=\int_\Omega \mathcal{C}(\bm{x},t,\bm{\varepsilon}(\bm{u}))\bm{\cdot}\bm{\varepsilon}(\bm{v})d\Omega,\\
F(t,w,\bm{u})=\alpha(t)p(w,u_\nu),\\
\phi(t,\bm{u},\zeta)=\lambda_{D}\left(\frac{1-\zeta}{\zeta}\right)-\frac{1}{2} \lambda_{E}\|\bm{\varepsilon}(\boldsymbol{u})\|^{2}+\lambda_{w},\\
j(w,\bm{r}_1,\bm{r}_2)=\int_{\Gamma_3}\mu p(w,{r_1}_\nu) \bm{n^*}{\bm{r}_2}_{\tau} d\Gamma+\int_{\Gamma_3}p(w,{r_1}_\nu) {r_2}_{\nu} d \Gamma
\end{cases}
\end{align*}
for all $\bm{u},\;\bm{v}\in V$, $\bm{r}_1,\;\bm{r}_2\in X$, $w\in W$ and $\zeta\in Y$. Then Problem \ref{problem_1} can be transformed as follows
\begin{align*}
&\dot{w}(t)=F(t,w(t),\bm{u}(t)),\\
&\left\langle A(t,{\bm{u}}(t)) +\int_0^t B(t-s,{\bm{u}}(s),\zeta(s))ds+C(t,\dot{\bm{u}}(t)),\bm{v}-\dot{\bm{u}}(t) \right\rangle_{V^*\times V}\nonumber\\
& \qquad +j(w(t) ,\dot{\bm{u}},\bm{v})-j(w(t),\dot{\bm{u}}(t) ,\dot{\bm{u}}(t))\geq \langle f(t),\bm{v}-\dot{\bm{u}}(t) \rangle_{V^*\times V},\quad\forall \bm{v}\in K_V ,\\
&\langle \dot{\zeta}(t) ,\eta-\zeta(t) \rangle_{Y_1}+a(\zeta(t) ,\eta-\zeta(t) )\geq\langle\phi(t,u(t) ,\zeta(t) ),\eta-\zeta(t) \rangle_{Y_1},\quad\forall\eta\in K_Y,\\
&\bm{u}(0)=\bm{u}_0,w(0)=0,\zeta(0)=\zeta_0\in(0,1).
\end{align*}

Next we show that assumption conditions H(A)-H(a) are satisfied.

I). It is easy to see that $A(\cdot,\cdot)$ is continuous on $I\times V$ for all $t\in I$. From \eqref{assume:A}(c) and the Holder's inequality, we have
\begin{eqnarray*}
\langle A(t,\bm{u}_1)-A(t,\bm{u}_2),\bm{v}\rangle_{V^*\times V}&\leq&\left(\int_\Omega \|\mathcal{A}(\bm{x},t,\bm{\varepsilon}(\bm{u}_1))
-\mathcal{A}(\bm{x},t,\bm{\varepsilon}(\bm{u}_2))\|^2d\Omega\right)^{\frac{1}{2}}\|\bm{v}\|_V\\
&\leq&L_\mathcal{A}\left(\int_{\Omega}\|\bm{\varepsilon}(\bm{u}_1)-\bm{\varepsilon}(\bm{u}_2)\|^2
d\Omega\right)^{\frac{1}{2}}\|\bm{v}\|_V
\end{eqnarray*}
for all $\bm{u}_1,\;\bm{u}_2,\;\bm{v}\in V$ with $t\in I$. It follows that
$$
\| A(t,\bm{u}_1)-A(t,\bm{u}_2)\|_{V^*}\leq L_\mathcal{A} \|\bm{u}_1-\bm{u}_2\|_V.
$$
This shows that $L_A= L_\mathcal{A} $ in \eqref{assume:A}(c) and so assumption condition H(A) holds.

II). Condition H(B)(a) can be verified from \eqref{assume:B}(b). In addition, by H\"{o}lder's inequality and \eqref{assume:B}(c), one has
\begin{eqnarray*}
\langle B(t,\bm{u}_1,\zeta)-B(t,\bm{u}_2,\eta),\bm{v}\rangle_{V^*\times V}&\leq&\left(\int_\Omega \|\mathcal{B}(\bm{x},t,\bm{\varepsilon}(\bm{u}_1),\zeta)
-\mathcal{B}(\bm{x},t,\bm{\varepsilon}(\bm{u}_2),\eta)\|^2d\Omega\right)^{\frac{1}{2}}\|\bm{v}\|_V\\
&\leq&\sqrt{2}L_\mathcal{B}\left(\int_{\Omega}\|\bm{\varepsilon}(\bm{u}_1)-\bm{\varepsilon}(\bm{u}_2)\|^2
+\|\zeta-\eta\|^2d\Omega\right)^{\frac{1}{2}}\|\bm{v}\|_V
\end{eqnarray*}
and
\begin{eqnarray*}
\langle B(t,\bm{u},\zeta),\bm{v}\rangle_{V^*\times V}&\leq&\left(\int_\Omega\|\mathcal{B}(\bm{x},t,\bm{\varepsilon}(\bm{u}),\zeta)\|^2d\Omega\right)^{\frac{1}{2}}\|\bm{v}\|_V\\
&\leq&\sqrt{2}\rho_{\mathcal{B}}(t)\left(\int_{\Omega}\|\bm{\varepsilon(u)}\|^2+\|\zeta\|^2d\Omega\right)^{\frac{1}{2}}\|\bm{v}\|_V
\end{eqnarray*}
for all $\bm{u}_1,\;\bm{u}_2,\;\bm{u},\;\bm{v}\in V$ and $\zeta,\;\eta\in Y$ with $t\in I$.  It follows that
$$
\begin{cases}
\| B(t,\bm{u}_1,\zeta)-B(t,\bm{u}_2,\eta)\|_{V^*}\leq\left(\sqrt{2}L_\mathcal{B}\right) \left(\|\bm{u}_1-\bm{u}_2\|_V+\|\zeta-\eta\|_Y\right),\\
\| B(t,\bm{u},\zeta)\|_{V^*}\leq\left(\sqrt{2}\rho_{\mathcal{B}}(t)\right) \left(\|\bm{u}\|_V+\|\zeta\|_Y\right).
\end{cases}
$$
Thus, $L_B= \sqrt{2}L_\mathcal{B}  $ and $\rho(t)= \sqrt{2}\rho_{\mathcal{B}}(t) $ in H(B)(b)(c) and so assumption condition H(B) is true.

III). It is easy to see that $C(\cdot,\cdot)$ is continuous on $I\times V$ for all $t\in I$. From \eqref{assume:C}(c)(d)(e) and H\"{o}lder's inequality, we have
\begin{eqnarray*}
\langle C(t,\bm{u})-C(t,\bm{v}),\bm{u}-\bm{v}\rangle_{V^*\times V}&=&\int_\Omega \left(\mathcal{C}(\bm{x},t,\bm{\varepsilon}(\bm{u}))-\mathcal{C}(\bm{x},t,\bm{\varepsilon}(\bm{v}))\right)
\bm{\cdot}\left(\bm{\varepsilon}(\bm{u})-\bm{\varepsilon}(\bm{v})\right)d\Omega\\
&\geq&m_\mathcal{C}\int_{\Omega}\|\bm{\varepsilon}(\bm{u})-\bm{\varepsilon}(\bm{v})\|^2d\Omega\\
&=&m_\mathcal{C}\|\bm{u}-\bm{v}\|^2_V,\\
\langle C(t,\bm{u}_1)-C(t,\bm{u}_2),\bm{v}\rangle_{V^*\times V}&\leq&\left(\int_\Omega \|\mathcal{C}(\bm{x},t,\bm{\varepsilon}(\bm{u}_1))
-\mathcal{C}(\bm{x},t,\bm{\varepsilon}(\bm{u}_2))\|^2d\Omega\right)^{\frac{1}{2}}\|\bm{v}\|_V\\
&\leq&L_{\mathcal{C}1}\left(\int_{\Omega}\|\bm{\varepsilon}(\bm{u}_1)-\bm{\varepsilon}(\bm{u}_2)\|^2
d\Omega\right)^{\frac{1}{2}}\|\bm{v}\|_V,\\
\langle C(t_1,\bm{u})-C(t_2,\bm{u}),\bm{v}\rangle_{V^*\times V}&\leq&\left(\int_\Omega \|\mathcal{C}(\bm{x},t_1,\bm{\varepsilon}(\bm{u}))
-\mathcal{C}(\bm{x},t_2,\bm{\varepsilon}(\bm{u}))\|^2d\Omega\right)^{\frac{1}{2}}\|\bm{v}\|_V\\
&\leq&L_{\mathcal{C}2}\left(\int_{\Omega}\|t_1-t_2\|^2
d\Omega\right)^{\frac{1}{2}}\|\bm{v}\|_V
\end{eqnarray*}
for all $\bm{u}_1,\;\bm{u}_2,\;\bm{u},\;\bm{v}\in V$ with $t,t_1,t_2\in I$.  It follows that
$$
\| C(t,\bm{u}_1)-C(t,\bm{u}_2)\|_{V^*}\leq L_{\mathcal{C}1} \left(\|\bm{u}_1-\bm{u}_2\|_V\right),\quad
\| C(t_1,\bm{u})-C(t_2,\bm{u})\|_{V^*}\leq L_{\mathcal{C}2} \left(\|t_1-t_2\|_V\right).
$$
Thus, $L_{C1}= L_{\mathcal{C}1}$, $L_{C2}= L_{\mathcal{C}2}$ and $m_C=m_{\mathcal{C}}$ in H(C)(a)(b)(c) and so assumption condition H(C) is satisfied.

IV). Since
$$j(w,\bm{r}_1,\bm{r}_2)=\int_{\Gamma_3}\mu p(w, {r_1}_\nu) \bm{n^*}{\bm{r}_2}_\tau +  p(w, {r_1}_\nu) {r_2}_\nu d \Gamma,$$
we know that $j(w,\bm{r}_1,\cdot)$ is a convex proper and lower semicontinuous functional with respect to $\bm{r}_2$.  Moreover, the convex subdifferential of $j(w,\bm{r}_1,\bm{r}_2)$ with respect to its third variable can be given by
$$\partial j(w,\bm{r}_1,\bm{r}_2)=\int_{\Gamma_3}\mu p(w, {r_1}_\nu) \bm{n^*}+p(w, {r_1}_\nu)\bm{\nu}d\Gamma.$$
This fact combined with \eqref{assume:p}(b)(c) yields that
\begin{eqnarray*}
\|\partial j(w,\bm{r}_1,\bm{r}_2)\|_{X^*}&=&\left(\int_{\Gamma_3}\mu^2 p^2( {w,r_1}_\nu)+ p^2(w, {r_1}_\nu)d\Gamma\right)^\frac{1}{2}\nonumber\\
&\leq&\sqrt{1+\|\mu\|_{L^\infty(\Gamma_3;\mathbb{R})}^2} \left(\int_{\Gamma_3}(p( {w,r_1}_\nu)-p(0,0))^2 d\Gamma\right)^\frac{1}{2}\nonumber\\
&\leq&\sqrt{1+\|\mu\|_{L^\infty(\Gamma_3;\mathbb{R})}^2}\widetilde{L}_p \left(\int_{\Gamma_3}(|w|+|{r_1}_\nu|)^2 d\Gamma\right)^\frac{1}{2}\nonumber\\
&\leq&\sqrt{2+2|\mu\|_{L^\infty(\Gamma_3;\mathbb{R})}^2}\widetilde{L}_p(\|\bm{r}_1\|_X+\|w\|_W).
\end{eqnarray*}
Taking $\sqrt{2+2\|\mu\|_{L^\infty(\Gamma_3;\mathbb{R})}^2}\widetilde{L}_p=c_1$,  there exists a constant $c_1>0$ such that
\begin{eqnarray}\label{j}
\|\partial j(w,u,v)\|_X\leq c_1(\|w\|_W+\|u\|_X), \quad \forall (u,v,w)\in X\times X\times W.
\end{eqnarray}
Letting $w_1,\;w_2\in W$ and $\bm{s}_1,\;\bm{s}_2\in X$, one has
\begin{eqnarray}\label{j_ass1}
&&j(w_1,\bm{r}_1,\bm{s}_2)-j(w_1,\bm{r}_1,\bm{s}_1)+j(w_2,\bm{r}_2,\bm{s}_1)-j(w_2,\bm{r}_2,\bm{s}_2)\nonumber\\
&=&\int_{\Gamma_3}\left[\mu p( w_1,{r_1}_\nu) \bm{n^*}({\bm{s}_2}_\tau -{\bm{s}_1}_\tau )+ p(w_1,{r_1}_\nu)( {s_2}_\nu - {s_1}_\nu )+\mu p( w_2,{r_2}_\nu) \bm{n^*}({\bm{s}_1}_\tau -{\bm{s}_2}_\tau )\right.\nonumber\\
&& \mbox{}\left. + p(w_2,{r_2}_\nu)( {s_1}_\nu - {s_2}_\nu )\right]d \Gamma\nonumber\\
&=&\int_{\Gamma_3}\mu \left(p(w_1,{r_1}_\nu)-p(w_2,{r_2}_\nu)\right)\bm{n^*}({\bm{s}_2}_\tau -{\bm{s}_1}_\tau )d\Gamma+\int_{\Gamma_3}( p( w_1,{r_1}_\nu)-p( w_2,{r_2}_\nu))( {s_2}_\nu - {s_1}_\nu )d \Gamma\nonumber\\
&\leq&\int_{\Gamma_3}|\mu| |p(w_1,{r_1}_\nu)-p(w_2,{r_2}_\nu)|\|{\bm{s}_2}_\tau -{\bm{s}_1}_\tau \|d\Gamma+\int_{\Gamma_3} |p( w_1,{r_1}_\nu)-p( w_2,{r_2}_\nu)||{s_2}_\nu -{s_1}_\nu |d\Gamma\nonumber\\
&\leq&\|\mu\|_{L^\infty(\Gamma_3;\mathbb{R})}\int_{\Gamma_3} |p( w_1,{r_1}_\nu)-p( w_2,{r_2}_\nu)|\|{\bm{s}_2}_\tau -{\bm{s}_1}_\tau \|d\Gamma\nonumber\\ &&\mbox{} +\int_{\Gamma_3} |p( w_1,{r_1}_\nu)-p( w_2,{r_2}_\nu)||{s_2}_\nu -{s_1}_\nu |d\Gamma.
\end{eqnarray}
The condition \eqref{assume:p}(b) combined with H\"{o}lder's inequality shows that
\begin{eqnarray*}
&&\int_{\Gamma_3} |p( w_1,{r_1}_\nu)-p( w_2,{r_2}_\nu)|\|{\bm{s}_2}_\tau -{\bm{s}_1}_\tau \|d\Gamma\nonumber\\
&\leq&\tilde{L}_p\int_{\Gamma_3}(|{r_1}_\nu-{r_2}_\nu|+|w_1-w_2|)\|{\bm{s}_2}_\tau -{\bm{s}_1}_\tau \|d\Gamma\nonumber\\
&\leq&\tilde{L}_p(\|{r_1}_\nu-{r_2}_\nu\|_X+|w_1-w_2|_W)\|\bm{s}_2 -\bm{s}_1 \|_X
\end{eqnarray*}
and
\begin{eqnarray}\label{e4.34}
\int_{\Gamma_3} |p( w_1,{r_1}_\nu)-p( w_2,{r_2}_\nu)|\| {s_2}_\nu - {s_1}_\nu \|d\Gamma
\leq\tilde{L}_p(\|{r_1}_\nu-{r_2}_\nu\|_X+|w_1-w_2|_W)\|\bm{s}_2 -\bm{s}_1 \|_X.
\end{eqnarray}
Then inequality \eqref{j_ass1} becomes
\begin{eqnarray}\label{e5.13}
&&j(w_1,\bm{r}_1,\bm{s}_2)-j(w_1,\bm{r}_1,\bm{s}_1)+j(w_2,\bm{r}_2,\bm{s}_1)-j(w_2,\bm{r}_2,\bm{s}_2)\nonumber\\
&\leq&(\|\mu\|_{L^\infty(\Gamma_3;\mathbb{R})}+1)\tilde{L}_p\left(\|w_1-w_2\|_W\|\bm{s}_1-\bm{s}_2\|_X+\|\bm{r}_1-\bm{r}_2\|_X\|\bm{s}_1-\bm{s}_2\|_X\right).
\end{eqnarray}
This implies that H(j)(b) holds with
$$\alpha_0=(\|\mu\|_{L^\infty(\Gamma_3;\mathbb{R})}+1)\tilde{L}_p, \quad \alpha_1=(\|\mu\|_{L^\infty(\Gamma_3;\mathbb{R})}+1)\tilde{L}_p.$$
and so assumption condition H(j) is satisfied.

V). It follows from \eqref{e4.34} that
\begin{eqnarray*}
\|F(t,w_1,\bm{u}_1)-F(t,w_2,\bm{u}_2)\|^2_W&=&\int_{\Gamma_3}\alpha(t)^2(p(w_1,{u_1}_\nu)-p(w_2,{u_2}_\nu))^2d\Gamma\nonumber\\
&\leq& 2 \|\alpha\|_{C(I;L^\infty(\Gamma_3;\mathbb{R}))}^2\tilde{L}_p\left(\|\bm{u}_1-\bm{u}_2\|_X^2+\|w_1-w_2\|_W^2\right)
\end{eqnarray*}
and so
\begin{eqnarray*}
\|F(t,w_1,\bm{u}_1)-F(t,w_2,\bm{u}_2)\|_W\leq\sqrt{2\tilde{L}_p} \|\alpha\|_{C(I;L^\infty(\Gamma_3;\mathbb{R}))} (\|\bm{u}_1-\bm{u}_2\|_V+\|w_1-w_2\|_W).
\end{eqnarray*}
Thus, H(F)(b) holds with
$$L_p=\max\left\{\sqrt{2\tilde{L}_p}\|\alpha\|_{C(I;L^\infty(\Gamma_3;\mathbb{R}))}
,\sqrt{2\tilde{L}_p}\|\alpha\|_{C(I;L^\infty(\Gamma_3;\mathbb{R}))}\right\}.$$
Clearly,  H(F)(a) follows from condition \eqref{assume:n} and so assumption condition H(F) is fulfilled.

VI). Based on \eqref{assume:phi}, we can conclude that $H(\phi)$ holds with $L_\phi=\sqrt{2}\tilde{L}_\phi$.

VII). Finally, we can check that H(a) holds. Indeed, from
$$\|\zeta\|_Y^2=\|\zeta\|^2_{Y_1}+\int_\Omega \nabla\zeta\cdot\nabla\zeta d\Omega,$$
it follows that $a(\zeta,\zeta)+\kappa\|\zeta\|^2_{Y_1}=\kappa\|\zeta\|_Y^2$ and so H(a) is satisfied with $a_1=a_2=\kappa.$

Thus, combining I)-VII), we know that Theorem \ref{theorem_exist} is a direct consequence of Theorem \ref{t3.1}.
\end{proof}

\section*{Declarations}
{\bf Conflicts of interest/Competing interests} The authors have no known competing financial interests or personal
relationships that could have appeared to influence the work reported in this manuscript.


\begin{thebibliography}{99}

\bibitem{KT1997}
K.T. Andrews, M. Shillor, S. Wright, A. Klarbring and M. Sofonea, A dynamic thermoviscoelastic contact problem with friction and wear, International Journal of Engineering Science, {\bf35}(1997), 1291-1309.

\bibitem{T1999}
T. Angelov, On a rolling problem with damage and wear, Mechanics Research Communications, {\bf 26}(1999), 281-286.


\bibitem{Brog2020}
B. Brogliato, A. Tanwani, Dynamical systems coupled with monotone set-valued operators: formalisms applications, well-posedness, and stability, SIAM Review,
{\bf 62}(2020), 3-129.

\bibitem{Capatina2014}
A. Capatina, Variational Inequalities and Frictional Contact Problems, Springer, New York, 2014.

\bibitem{Chen2021}
T. Chen, N.J. Huang, X.S. Li, Y.Z. Zou, A new class of differential nonlinear system involving parabolic varitional and history-dependent hemi-variational inequalities arising in contact mechanics, Communications in Nonlinear Science and Numerical Simulation, {\bf101}(2021), 105886.

\bibitem{Chen2020}
T. Chen, N.J. Huang, Y.B. Xiao, Variational and numerical analysis of a dynamic viscoelastic contact problem with friction and wear, Optimization, {\bf69}(2020), 2003-2031.

\bibitem{Eck2010}
C. Eck, J. Jaru\u{s}ek, M. Sofonea, A dynamic elastic-visco-plastic unilateral contact problem with normal damped response and Coulomb friction, European Journal of Applied Mathematics, {\bf 21}(2010), 229-251.

\bibitem{ChenX2014}
X.J Chen, Z.Y. Wang, Differential variational inequality approach to dynamic games with shared consteaints, Mathematical Programming, {\bf146(1)}(2014), 379-408.

\bibitem{Fr1995}
M. Fr\'{e}mond, B. Nedjar, Damage in concrete: the unilateral phenomenon, Nuclear Engineering and Design, {\bf 156}(1995), 323-335.

\bibitem{Fr1996}
M. Fr\'{e}mond, B. Nedjar, Damage, gradient of damage and principle of virtual work, International Journal of Solids and Structures, {\bf 33}(1996), 1083-1103.

\bibitem{Gasi2015}
L. Gasi\'{n}ski, A. Ochal, M. Shillor, Variational-hemivariational approach to a quasistatic viscoelastic problem with normal compliance, friction and material damage, Zeitschrift f\"{u}r Analysis und ihre Anwendungen, {\bf 34}(2015), 251-276.

\bibitem{Gwinner2013}
J. Gwinner, On a new class of differential variational inequalities and a stability result, Mathematical Programming, {\bf 139(1-2)}(2013), 205-221.


\bibitem{Han2016}
J. Han, S. Mig\'{o}rski, A quasistatic viscoelastic frictional contact problem with multivalued normal compliance, unilateral constraint and material damage, Journal of Mathematical Analysis and Applyciations, {\bf 443}(2016), 57-80.

\bibitem{Han2001}
W. Han, M. Sofonea, Quasistatic Contact Problems in Viscoelasticity and Viscoplasticity, American Mathematical Society, Somerville, 2001.


\bibitem{A2018}
A Kasri, A Touzaline, A quasistatic frictional contact problem for viscoelastic materials with long memory, Georgian Mathematical Journal, {\bf 27(2)}(2020), 249-264.

\bibitem{AS2007}
A.S. Kravchuk, P.J. Neittaanm\"{a}ki, Variational and Quasi-variational Inequalities in Mechanics, Springer, Dordrecht, 2007.

\bibitem{Kulig2018}
A. Kulig, Variational-hemivariational approach to quasistatic viscoplastic contact problem with normal compliance, unilateral constraint, memory term, friction and damage, Nonlinear Analysis: Real World Applications, {\bf 44}(2018), 401-416.


\bibitem{LiW2015}
W. Li, X. Wang, N.J. Huang, Differential inverse variational inequalities in finite dimensional spaces, Acta Mathematica Scientia, {\bf 35}(2015), 407-422.

\bibitem{LiW2017}
W. Li, Y.B. Xiao, N.J. Huang, Y.J. Cho, A class of differential inverse quasivariational inequalities in finite dimensional spaces, Journal of Nonlinear Sciences and Applications, {\bf 10}(2017), 4532-4543.

\bibitem{LiXS2010}
X.S. Li, N.J. Huang, D. O'Regan, Differential mixed variational inequalities in finite dimensional spaces, Nonlinear Analysis: Theory Methods and Applications, {\bf 72}(2010), 3875-3886.

\bibitem{Liu2018}
Z. Liu, S. Mig\'{o}rski, S. Zeng, A class of history-dependent differential variational inequalities with application to contact problems, Optimization {\bf 69} (2020), 743-775.


\bibitem{Liu2021+}
Z.H. Liu, D. Motreanu, S.D. Zeng, Generalized penalty and regularization method for differential variational hemivariational
inequalities, SIAM Journal on Optimization, {\bf 31}(2021), 1158-1183.


\bibitem{LiuSofonea2019}
Z.H. Liu, M. Sofonea, Differential quasivariational inequalities in contact mechanics, Mathematics and Mechanics of Solids, {\bf 24}(2019), 845-861.

\bibitem{Liu2021}
Z.H. Liu, S.D. Zeng, Penalty method for a class of differential variational inequalities, Applicable Analysis, {\bf 100}(2021), 1574-1589.

\bibitem{MOS2013}
S. Mig\'{o}rski, A. Ochal, M. Sofonea, Nonlinear Inclusions and Hemivariational Inequalities, Springer, New York, 2013.

\bibitem{MOS2015}
S. Mig\'{o}rski, A. Ochal, M. Sofonea, History-dependent variational-hemivariational inequalities in contact mechanics, Nonlinear Analysis: Real World Applications,
{\bf 22}(2015), 604-618.

\bibitem{NP1994}
Z. Naniewicz, P.D. Panagiotopoulos, Mathematical Theory of Hemivariational Inequalities and Applications, CRC Press, New York, 1994.


\bibitem{Pang2008}
J.S. Pang, D.E. Stewart, Differential variational inequalities, Mathematical Programming, {\bf 113}(2008), 345-424.


\bibitem{Han2002}
M. Shillor, M. Sofonea, J.J. Telega, Models and Analysis of Quasistatic Contact, Springer,  Berlin, 2004.


\bibitem{Sofonea2011}
M. Sofonea, A. Matei, History-depemdent quasivariational inequalities arising in cantact mechanics, European Journal of Applied Mathematics, {\bf 22}(2011), 471-491.

\bibitem{Sofonea2012}
M. Sofonea, A. Matei, Mathematical Models in Contact Mechanics, Cambridge University Press, Cambridge, 2012.

\bibitem{SofoneaF2016}
M. Sofonea, F. P\u{a}trulescu, Y. Souleiman, Analysis of a contact problem with wear and unilateral constraint, Applicable Analysis, {\bf 95}(2016), 2590-2607.

\bibitem{Sofonea2016}
M. Sofonea, Y.B. Xiao, Fully history-dependent quasivariational inequalities in contact mechanics, Applicable Analysis, {\bf 95}(2016), 2464-2484.



\bibitem{Wang2017}
X. Wang, Y.W Qi, C.Q. Tao, Y.B. Xiao, A new class of delay differential variational inequalities, Journal of Optimization Theory and Applications, {\bf 172(1)}(2017), 56-69.

\bibitem{Weng2021}
Y.H. Weng, T. Chen, N.J. Huang, A new fracrional nolinear system driven by a quasi-hemivariational inequality with an application, Journal of Nonlinear and Convex Analysis, {\bf 22(3)}(2021), 559-586.

\bibitem{Weng2021+}Y.H. Weng, T. Chen, X.S. Li, N.J. Huang, Rothe method and numerical analysis for a new class of fractional differential hemivariational inequality with an application, Computers and Mathematics with Applications, {\bf 98}(2021), 118-138.

\bibitem{Zeng2018}
S.D. Zeng, Z.H. Liu, S. Mig\'{o}rski, A class of fractional differential hemivariational inequalities with application to contact problem, Zwitscgrift f\"{u}r Angewandte Mathematik und Physik, {\bf 69}(2018), 36.

\bibitem{Zeng2021}
S.D. Zeng, S. Mig\'{o}rski, A.A. Khan,  Nonlinear quasi-hemivariational inequalities: existence and optimal control, SIAM Journal on Control and Optimization, {\bf 59}(2021), 1246-1274.

\bibitem{Zeng2021+}S.D. Zeng, S. Mig\'{o}rski, Z.H. Liu, J.C. Yao, Convergence of a generalized penalty method for variational-hemivariational inequalities, Communications in Nonlinear Science and Numerical Simulation, {\bf 92}(2021), 105476.
\end{thebibliography}
\end{document}